\documentclass[10pt,a4paper,twoside]{amsart}

\usepackage{amsfonts, amssymb, amsmath, amsthm, bm}
\usepackage{mathrsfs} 
\usepackage{latexsym}
\usepackage{enumerate}
\usepackage{multicol}
\usepackage{verbatim}
\usepackage{dsfont}
\usepackage{subfig}

\usepackage{url}
\usepackage{csquotes}
\usepackage{graphicx}
\usepackage{epstopdf}

\usepackage[colorlinks=true]{hyperref}
\hypersetup{citecolor=blue, linkcolor=blue}
\usepackage{mathabx} 

\newtheorem{thm}{Theorem}[section]
\newtheorem{lem}[thm]{Lemma}
\newtheorem{cor}[thm]{Corollary}

\newcommand{\Z}[1]{\mathbb{Z}/#1\mathbb{Z}}

\def\order{\asymp}
\DeclareMathOperator{\Log}{log}

\DeclareMathOperator{\sgn}{sgn}

\def\Kcal{\mathcal{K}}

\def\Oc{\mathcal{O}}
\def\Ocal{\mathcal{O}}

\def\1{\mathds{1}}

\def\K{\mathcal{K}}

\def\mode{\mathbin{\,\textrm{mod}^*}}
\def\mod{\mathbin{\,\textrm{mod}\,}}

\DeclareMathOperator{\Si}{Si}
\DeclareMathOperator{\sinc}{sinc}
\newcommand{\myc}{\mathfrak{c}}

\title{The weighted large sieve through Parseval%
}
\author{Olivier Ramar\'e}
\address[O. Ramar\'e]{CNRS/ Institut de Math\'ematiques de Marseille, Aix 
 Marseille Universit\'e, U.M.R. 7373, Site Sud, Campus de Luminy, Case 907, 
 13288 
 Marseille Cedex 9, France.}
\email{olivier.ramare@univ-amu.fr}

\begin{document}

\subjclass[2010]{Primary: 11N13, 11N35, 11N36, Secondary: }

 \keywords{Large sieve inequality, Weighted large sieve}

\maketitle
\begin{abstract}
   We modify the approach to the arithmetical form of the large sieve
   by relying on the Parseval identity rather than on an approximate
   Bessel inequality and as a consequence, improve on the weighted
   large sieve inequality. We compute in passing the $L^2$-norm of the
   Beurling-Selberg majorant of the characteristic function of an interval. 
 \end{abstract}


\section{Introduction and results}

Since their discovery by V.~Brun in~\cite{Brun*19, Brun*19b}, 
sieve methods have proved very efficient to \emph{bound above} quantities
defined by the exclusion of some or several residue classes. Sieve
approaches can be classified in three main categories: the combinatorial
sieves, the Selberg (or $\Lambda^2$) sieve and the large
sieves. Gallagher's larger sieve introduced in~\cite{Gallagher*71}
remains an isolated tool, though a powerful one where it applies. We
are concerned in this paper with the large sieve that
arose from the work of Yu.\,V.~Linnik
in~\cite{Linnik*42}; we refer the readers to the
monographs~\cite{Montgomery*71} by H.\,L.~Montgomery
and~\cite{Ramare*06} for more historical details.
As explained by H.\,L.~Montgomery in~\cite{Montgomery*78}, the
denomination \emph{large sieve} or \emph{large sieve inequality} usually refers to an inequality on a
finite set of points. We proposed in~\cite{Ramare*26-1} a different
approach by relying on the Parseval identity on
$\mathbb{R}/\mathbb{Z}$ and recovered in this way a slightly weaker form of what
is known as \emph{the arithmetical form of the large sieve} or
\emph{Montgomery's sieve}; we also proved that the
trigonometric polynomial of a sifted sequence, if large, has to have a sharp peak
around the origin. See Lemma~\ref{ManypsL20} below.

In the present paper, we use the Parseval identity in a different
manner and obtain another arithmetical form of the weighted large
sieve inequality, which, in some cases, is stronger than the earlier one
contained in~\cite{Montgomery-Vaughan*73} by H.~Montgomery \&
R.\,C.~Vaughan. To do so, and as a sidedish, we compute in
Theorem~\ref{myNorm} the $L^2$-norm of the Beurling-Selberg majorant
of the characteristic function of an interval. In many emblematic
problems, this already enables us to sieve up to $Q=\sqrt{N}$. We next incorporate the
sharp-peak property in our method, obtaining a 
quantitative improvement of our result.

\subsection*{A challenge}
Let us start this paper by an upper bound for the number of twin primes in an
interval, i.e. a very special result
on which one may compare different methods.
A striking consequence of the Montgomery \& Vaughan weighted large
sieve is a sharp bound in this case, see \cite{Siebert*76} by H.~Siebert and
Lemma~5 in~\cite{Riesel-Vaughan*83} by H.~Riesel \&
R.\,C.~Vaughan. Here is what we may obtain now.
\begin{thm}
  \label{p+2}
  When $N$ is large enough, we have
  \begin{equation*}
    \sum_{\substack{M<p\le M+N\\ \text{$p+2$ is prime}}}\mkern-10mu1
    \le \frac{8C N}{(11.57+\log N)\log N},
    \qquad
    C = 2\prod_{p\ge3}\frac{p(p-2)}{(p-1)^2}=1.32032\cdots
  \end{equation*}
\end{thm}
\noindent
H.~Riesel \& R.\,C.~Vaughan obtained the same bound though with $8.45$
rather than~$11.57$.  Let us recall that, in this situation, the
Selberg sieve leads only to the weaker denominator
$(\log N)^2-\Ocal(\log N\log\log N)$, as is established
in Theorem~3.11 of the book~\cite{Halberstam-Richert*74} by
H.~Halberstam \&{} H.-E.~Richert.  The case $M=0$ may be treated
very differently. It has been the subject
of very refined studies and the coefficient~8 has for instance been
reduced to~$3.3996$ by J.~Wu in~\cite{Wu*04}. Further
improvements have maybe taken place in between, while $1+o(1)$ is
expected to be admissible. These methods do not apply when $M$ is
large with respect to~$N$.

\subsection*{A first general result}
In the classical sieve setting, we start from a sequence of (excluded)
subsets $\mathcal{L}_p\subset\Z{p}$ for each prime $p\le Q$, for
some~$Q$, and consider the integers~$n$ from a host sequence, say an
interval~$[M+1,M+N]$, that are such that, for every $p\le Q$, we have
$n\notin\mathcal{L}_p$. Here is a first corollary of our method.
\begin{cor}
  \label{easycor}
  Let $Z$ be the number of integers $n\in[M,M+N]$ that do not belong to any of
  the sets $\mathcal{L}_p$, for $p\le Q$. We have
  \begin{equation*}
    Z\le 1/L^*(Q)
  \end{equation*}
  where
  \begin{equation}
    \label{defLstar}
    L^*(Q) = \sum_{q\le Q}
    \frac{\mu^2(q)}{N+\myc q(q+Q)}
    \prod_{p|q}\frac{|\mathcal{L}_p|}{p-|\mathcal{L}_p|}
  \end{equation}
  where $\myc=0.898\,678\cdots$ is
  defined in \eqref{defc}.
\end{cor}
\noindent
This may be compared with Corollary~1
of~\cite{Montgomery-Vaughan*73} by H.~Montgomery \& R.\,C.~Vaughan where
the coefficient~$\myc q(q+Q)$ is replaced
by~$\frac32 qQ$.
Our expression is better when~$q$ is small but worse when~$q$ is
large. Two examples are examined in Section~\ref{CaseStudies}: when tried on the Brun-Titchmarsh Theorem, we obtain a result
of lesser quality than the one in~\cite{Montgomery-Vaughan*73}, while
when tried on the case of twin primes, we get a better result, see
Section~\ref{CaseStudies}.
E.~Preissmann in~\cite{Preissmann*84} improved this $3/2$ to
$\sqrt{1+\frac23\sqrt{6/5}}=1.315\dots$ and this appears to be the
last improvement on this matter, though the constant~$1$ is
conjectured to be admissible.
Montgomery \& Vaughan
in~\cite{Montgomery-Vaughan*73} use a weighted version of
Hilbert's inequality. We refer the readers to \cite{Yangjit*23} by
W.\,Yangjit \index{Yangjit@Yangjit,~Wijit} and to
\cite{Carneiro-Littman*24} by E.\,Carneiro \& F.\,Littmann for recent
work on this inequality.
If the weighted large sieve inequality were true with the
$3/2$ reduced to~1, we could modify the proof of Lemma~5
of~\cite{Riesel-Vaughan*83}, select $z=\sqrt{x}$, and conclude that
the 11.45 of Theorem~\ref{p+2} could only be replaced by $A_6(1)=2A_4-4\log
2=9.27\cdots$, in this paper's notation,

Let us finally recall that improving on Corollary~\ref{easycor} may have
consequences on the location of the possible Siegel zero, see for
instance~\cite{Motohashi*79} by Y.\,Motohashi or
\cite{Ramachandra-Sankaranarayanan-Srinivas*96} by
K.~Ramachandra, A.~Sankaranarayanan \&{} K.~Srinivas.

\subsection*{On the method}
As explained in detail in~\cite{Ramare*26-1}, the arithmetical usage
of the large sieve relies on a local (i.e. modulo~$q$) lower bound
coupled to a global upper bound. 
We equally rely on an equality for the lower bound, but one that is stronger
than the usual $L^2$ lower bound, though still very close to it.
As for the upper bound, and as in~\cite{Ramare*26-1}, we rely on the
Farey dissection of the unit interval, as in the circle method.
Coupled with a result (Theorem~\ref{myNorm}) on a smoothing function,
Theorem~\ref{WeightedSieve} (and consequently Corollary~\ref{easycor})
will follow.

This first proof is rather self-contained, and exhibits the
basic mechanism. In a second step, and on using stronger regularity
assumptions on our sieving conditions, we improve on this first result
by using the $L^2$-bound proved in the first paper~\cite{Ramare*26-1}
of this series.

\subsection*{A more general result}
We presented sieving in the first paragraph through an exclusion
hypothesis. This approach, though natural and historically older, 
leads to difficulties. It is better to say that
$n\in\Kcal_p=\Z{p}\setminus \mathcal{L}_p$.
This is for instance the
viewpoint adopted by M.\,N.~Huxley in his book~\cite{Huxley*72-2} and
by the present author in~\cite{Ramare*06, Ramare*10, Ramana-Ramare*25}.

Let us present the sieving situation from scratch.  A subset
$\K_q\subset\Z{q}$ is said to be \emph{multiplicative}\footnote{In
  earlier works, I used \emph{multiplivatively split} instead of the
  simpler \emph{multiplicative}.} if, when the decomposition of $q$ in
prime factors reads
\begin{equation*}
  q=p_1^{e_1}p_2^{e_2}\cdots p_r^{e_r},\quad (\forall i\neq j, \quad p_i\neq p_j),
\end{equation*}
and the Chinese Remainder Map is defined by
\begin{equation*}
  \sigma:
  \begin{array}[t]{rcl}
    \Z{q}&\rightarrow&{\displaystyle \prod_{1\le i\le r}\Z{p_i^{e_i}}}\\
    x&\mapsto&\bigl(x\mod p_i^{e_i}\bigr),
  \end{array}
\end{equation*}
we have the property
$
  \sigma^{-1}\bigl(\sigma\bigl(\K_q\bigr)\bigr)=\K_q$.
This is often written in the shorter form
 $ \K_q=\prod_{1\le i\le r}\K_{p_i^{e_i}}$.
We further say that the sequence $(\Kcal_q)_{q\le Q}$ is
\emph{consistent} when $\Kcal_q/d\mathbb{Z}=\Kcal_d$ whenever $d|q$.
It may be expedient to restrict the modulus~$q$ to square-free values,
it would solely be more difficult to write, as we should write
\begin{equation*}
  (\Kcal_q)_{q\le Q, \text{$q$ square-free}}.
\end{equation*}
We finally say that the \emph{Johnsen-Gallagher condition} 
  holds
  whenever
  \begin{equation}
    \label{jonhsen}
    \forall d|q,
    \forall y\in\K_d,\quad\#\{x\in\K_q:x\equiv y[d\}=|\K_q|/|\K_d|.
  \end{equation}
  This is equivalent to saying that the number of preimages in $\K_q$
  of any point~$y$ of~$\K_d$ does not depend on~$y$.
When $q$ is square-free and $\K_q$ is multiplicative, this condition
always holds.

Given a consistent multiplicative sequence $(\Kcal_q)_{q\le Q}$, we define
\begin{equation}
  \label{defg}
  g(q)=\prod_{p^\alpha\|q}
  \biggl(\frac{p^\alpha}{|K_{p^\alpha}|}
  -\frac{p^{\alpha-1}}{|K_{p^{\alpha-1}}|}\biggr).
\end{equation}

A sequence $(u_n)_{M<n\le M+N}$ of complex numbers is said \emph{have
support on $(\Kcal_q)_{q\le Q}$} whenever, when $u_n\neq0$, then $n$
belongs to every $\Kcal_q$ for $q\le Q$.
We may now state our main theorem.
\begin{thm}
  \label{WeightedSieve}
  Let $(u_n)_{M<n\le M+N}$ be a sequence of non-negative real numbers
  having support on a consistent multiplicative sequence
  $(\Kcal_q)_{q\le Q}$ that satisfies the Johnsen-Gallagher
  condition. We have
  \begin{equation*}
    \biggl| \sum_{n\le N}u_n\biggr|^2\le
    \sum_{n\le N}u_n^2/L(Q)
  \end{equation*}
  where
  \begin{equation}
    \label{defL}
    L(Q) = \sum_{q\le Q}\frac{g(q)}{N+\myc q(q+Q)}.
  \end{equation}
  The function $g$ is defined at~\eqref{defg} and the constant~$\myc$ is
  defined at~\eqref{defc}. 
\end{thm}
Corollary~\ref{easycor} is an immediate consequence of this result.
We presented the sieving situation by using consistent sequences and
the Johnsen-Gallagher condition. An approach closer to the one
employed by~A.~Selberg in~\cite{Selberg*76} is proposed
in~\cite{Jha-Ramana-Ramare*25}. Notice that the non-negativity is
superfluous as we may apply the result to~$(|u_n|)$ rather than to
$(u_n)$. We can even allow complex values.

\subsection*{A stronger result with more restrictive conditions}
We now come to the sharpest result of this paper. We state and prove
it on integer subsets and not on sequences as before to hopefully see
more clearly its scope. The previous setting is better when
$Q/\sqrt{N}$ remains small enough, and we want to select $Q$ as large
as possible. The question addressed is maybe best explained
on an example. Theorem~\ref{WeightedSieve} on the Brun-Titchmarsh
Theorem when $q=1$ gives an upper bound of the shape $2N/(\log N+C)$;
we are interested in increasing $C$. The property we use is that,
when this first upper bound is of the good order of magnitude, then
the set is regular, a property embedded in Corollary~2.6 of
\cite{Ramare*26-1} which is recalled below in Lemma~\ref{ManypsL20}.
\begin{thm}
  \label{better}
  Let $\mathcal{Z}\in[M,M+N]$ be a sequence of integers
  having support on a consistent multiplicative sequence
  $(\Kcal_q)_{q\le Q}$ that satisfies the Johnsen-Gallagher
  condition for $Q\le \pi \sqrt{N}/2$.
  Assume that $L(Q)$ (defined in~\eqref{defLstarstar}) satisfies
  $L(Q)=C (\log Q)^\kappa(1+\Ocal(1/\log Q))$, for some
  $C>0$ and $\kappa\ge0$ and that
  \begin{equation*}
    \max_{D\le Q}\sum_{\substack{p\ge 2,\nu\ge1\\ p^\nu\le
        D}}dg(d)/D\ll 1.
  \end{equation*}
  Then we have
  \begin{equation*}
    |\mathcal{Z}|\le
    \min\biggl(
    \frac{N}{L^{**}(Q)}
    \biggl(1+\Ocal\biggl(\frac{\sqrt{\log\log N}}{(\log
      N)^{3/2}}\biggr)\biggr)
    ,
    \frac{N}{L(\sqrt{N}(\log N)^{\frac52+\kappa})}
    \biggr)
  \end{equation*}
  where
  \begin{equation}
    \label{defLstarstar}
    L^{**}(Q) = \sum_{q\le Q}g(q)\biggl(1-\frac{2q(q+Q)}{\pi^2N}\biggr)
  \quad\text{and}\quad
    L(Q) = \sum_{q\le Q}g(q).
  \end{equation}
  The function $g$ is defined at~\eqref{defg}.
\end{thm}
Under mild hypotheses, the second term in the minimum can be skipped.

\begin{thm}
  \label{bettere}
  Let $\mathcal{Z}\in[M,M+N]$ be a sequence of integers
  having support on a consistent multiplicative sequence
  $(\Kcal_q)_{q\le Q}$ that satisfies the Johnsen-Gallagher
  condition for $Q\le \pi \sqrt{N}/2$.
  Assume that $L(Q)$ (defined in~\eqref{defLstarstar}) satisfies
  $L(Q)=C ((\log Q)^\kappa+c_0+o(1))(\log Q)^{\kappa-1})$, for some
  $C>0$ and $\kappa\ge1$ and that
\begin{equation*}
  \left\{
    \begin{array}{l}
      \displaystyle
      \sum_{\substack{ p\ge2, \nu\ge1\\ p^{\nu}\le D}}
      p^{\nu}g\bigl(p^{\nu}\bigr)\Log\bigl(p^{\nu}\bigr)=
      \kappa D+\Oc(D/\Log(2D)),\qquad(D\ge1),
      \\\displaystyle
      \sum_{p\ge2}
      \sum_{\substack{\nu,k\ge 1\\ p^{\nu+k}\le D}}
      p^{k+\nu}g\bigl(p^k\bigr)g\bigl(p^{\nu}\bigr)\Log\bigl(p^{\nu}\bigr)
      \ll \sqrt{D},
    \end{array}
  \right.
\end{equation*}
  Then, when $N$ is large enough, we have
  \begin{equation*}
    |\mathcal{Z}|\le
    \frac{2^\kappa N}{(\log N)^\kappa +2(c_0-0.238\kappa)(\log N)^{\kappa-1}}.
  \end{equation*}
\end{thm}

\subsection*{Acknowledgement}
The present paper was started when the author was enjoying the
hospitality of the Anhui university in China, completed when he was enjoying the hospitality of the International Centre for
Theoretical Sciences (ICTS) in Bengaluru, India during the program - The
Classical Circle Method and the Large Sieve 2026 (code:
ICTS/CCMLSS2026/05), and finalized when he benefited from the
hospitality of the Vilnius university. Let these institutions and
people involved be warmly thanked.
Most of the computations have been run by Pari/GP~\cite{PARI-GP} and,
for the figures, either the data in \texttt{svg}-format have been
piped to Inskscape~\cite{Inkscape} for 
functions, or to Sage~\cite{sagemath} to obtain the plots.

\section{A local lower bound}
This section is devoted to proving the next lower bound. This theorem
is new only in its generality. See for instance the
proof of Theorem~6 in the book~\cite{Bombieri*74} by E.~Bombieri or
Eq.\,(8.8) in the book~\cite{Huxley*72-2} by M.\,N.~Huxley. 

Let us define, as per Eq.\,(11.13) from~\cite{Ramare*06}, the function:
\begin{equation}
  \label{defpsistarq}
  \psi^*_q(n)=\sum_{\substack{\delta|q\\ n\in\Kcal_\delta}}\mu(q/\delta)
  \frac{\delta}{|\Kcal_\delta|}.
\end{equation}
Let us collect some of its properties in a lemma.
\begin{lem}
  \label{proppsistarq}
  We have
  $\displaystyle
    \sum_{\delta|q}\psi^*_\delta(n)=\frac{q}{|\Kcal_q|}\1_{\Kcal_q}(n)=\psi_q(n),
    $
  say. Furthermore we have
  \begin{equation}
  \label{eq:28}
  \psi^*_q(n)
  =
  \sum_{a\mode q}
  \hat{\psi}^*_q(a)e(-na/q)
  \quad\text{where}\quad
  \hat{\psi}^*_q(a)
  =\frac{1}{|\Kcal_q|}
  \sum_{c\in\Kcal_q}e(ac/q).
\end{equation}
\end{lem}
The function $\psi_q$ is the \emph{local model modulo~$q$} for our
sequence and $\psi^*_q$ is its ``new'' part.
It is readily shown that $\sum_{a\mode q}|\hat{\psi}^*_q(a)|^2=g(q)$,
and therefore our equality implies the inequality
\begin{equation*}
  \label{eq:31}
  \sum_{a\mode q}|S(a/q)|^2
  \ge
  g(q)
  |S(0)|^2
\end{equation*}
On reading carefully Chapter~8, and precisely Eq\,(8.4) therein, of
the book~\cite{Huxley*72-2} by M.\,N.~Huxley, the reader will discover
the very same (fundamental) identity.
\begin{proof}
  We proceed directly:
  \begin{equation*}
    \sum_{\delta|q}\psi^*_\delta(n)
    =\sum_{\delta|q}\sum_{\substack{t|\delta\\
    n\in\Kcal_t}}\mu(\delta/t)
    \frac{t}{|\Kcal_t|}
    =\sum_{\substack{t|q\\ n\in\Kcal_t}}
    \frac{t}{|\Kcal_t|}
    \sum_{\substack{\delta:t|\delta|q}}\mu(\delta/t)
    =\frac{q}{|\Kcal_q|}\1_{\Kcal_q}(n)
  \end{equation*}
  since the inner sum over $\delta$ vanishes when $t\neq q$, proving
  our first formula.
  Let us now express $\psi_q$ by its Fourier decomposition modulo~$q$. We
find that
\begin{align*}
  \psi_q(n)
  &=
    \frac{1}{q}\sum_{\delta|q}\frac{q}{|\Kcal_q|}\sum_{a\mode \delta}
    \sum_{b\in\Kcal_q}e(ab/\delta)e(-na/\delta)
  \\&=
    \sum_{\delta|q}\frac{1}{|\Kcal_\delta|}\sum_{a\mode \delta}
    \sum_{c\in\Kcal_\delta}e(ac/\delta)e(-na/\delta).
\end{align*}
By identification, or by using the Moebius inversion formula, we infer
that the summand in~$\delta$ is in fact $\psi^*_\delta(n)$. We thus get~\eqref{eq:28}.
This quantity computation may also be found in the paper~\cite{Kobayashi*73} by
I.\,Kobayashi (this is~$b_{q,a}$ therein).
The proof of our lemma is complete.
\end{proof}

\begin{thm}
  \label{L2S0}
  When $(u_n)$ has support on the consistent and multiplicative
  sequence~$(\Kcal_d)_{d|q}$ which satisfies the Johnsen-Gallagher
  condition, and with $S(\alpha)=\sum_n
  u_ne(n\alpha)$, we have
  \begin{equation}
    \sum_{a\mode q}\hat{\psi}^*_q(a)\overline{S(a/q)} = g(q)\overline{S(0)}.
  \end{equation}
\end{thm}


\begin{proof}

Notice the following important property:
\begin{center}
  \sl When $n\in\Kcal_q$, we have
  $\displaystyle\psi^*_q(n)=
  g(q)$.
\end{center}
We readily compute
\begin{equation*}
  \sum_{a\mode q}\hat{\psi}^*_q(a)\overline{S(a/q)}
  =\sum_{n}\overline{u_n}\sum_{a\mode q}\hat{\psi}^*_q(a) e(-na/q)
    =\sum_{n}\overline{u_n}\psi^*_q(n)=g(q)\overline{S(0)}
\end{equation*}
by Lemma~\ref{proppsistarq}. The proof is complete.
\end{proof}

\section{A special function}
\label{sectionM}

Let us follow the classical paper \cite{Vaaler*85} by J.~Vaaler.
We first introduce some notation and define\footnote{The Fourier
  transform $\hat{f}$ is defined by $\hat{f}(t)=\int_{-\infty}^\infty f(u)e(ut)du$.}\begin{equation}
  \label{defJhat}
  \hat{J}(t)
  =
  \begin{cases}
    1&\text{if $t=0$},\\
    \pi t(1-|t|)\cot\pi t+|t|&\text{if $0<|t|<1$},\\
    0&\text{if $1\le |t|$}.
  \end{cases}
\end{equation}
As per \cite[Theorem 6]{Vaaler*85},
the function $\hat{J}$ is even, non-negative, continuously
differentiable, and strictly decreasing on~$[0,1]$.
\begin{figure}[!h]
  \includegraphics[scale=0.8]{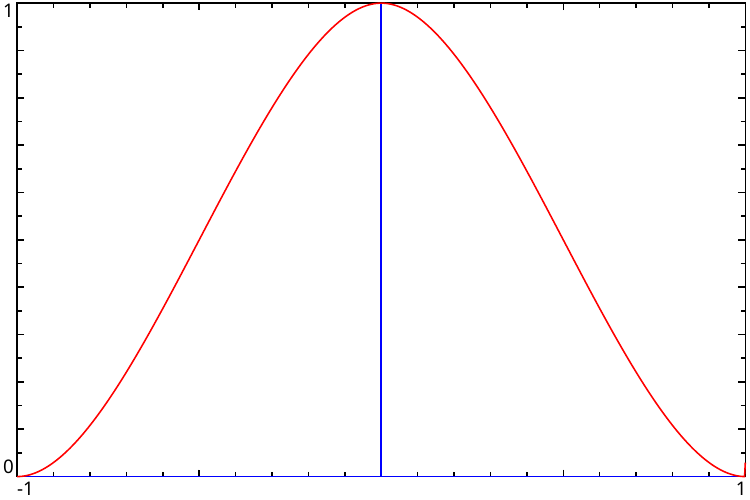}
  \caption{Plot of $\hat{J}(t)$}
  \label{PlotJhat}
\end{figure}
We continue with the definitions from the same paper:
\begin{equation}
  \label{eq:2}
  K(z)=\biggl(\frac{\sin\pi z}{\pi z}\biggr)^2,
  H(z)=\biggl(\frac{\sin\pi z}{\pi }\biggr)^2
  \biggl(\sum_{m\in\mathbb{Z}}\frac{\sgn(m)}{(z-m)^2}
  +\frac{2}{z}\biggr)
  \end{equation}
  (with $\sgn0=1$), and the Beurling-Selberg function
  $B(x)=K(x)+H(x)$. We finally set
  \begin{equation}
    \label{defCabdelta}
  2C_{[a,b],\delta}(x)
  =
  B(\delta(b-x))
  +B(\delta(x-a))\ge2\cdot\1_{x\in[a,b]}.
\end{equation}
\begin{lem}
  \label{BS}
  Let $\delta>0$ and $M\ge0$ be two parameters.
  The function $C_{[-M,M],\delta}$ is an upper bound for the
  characteristic function of $[-M,M]$.
  When $|t|\le \delta$, we have
  \begin{equation*}
    \hat{C}_{[-M,M],\delta}(t)
    =
    \delta^{-1}(1-|\delta^{-1}t|)\cos 2\pi M t
    +\frac{\hat{J}(\delta^{-1}t)}{\pi t}\sin 2\pi M t.
  \end{equation*}
  When $|t|\ge\delta$, we have $\hat{C}_{[-M,M],\delta}(t)=0$.
  We have
  \begin{enumerate}
  \item $\forall t\in\mathbb{R}$,\quad $|\hat{C}_{[-M,M],\delta}(t)|\le
    \hat{C}_{[-M,M],\delta}(0)=2M+\delta^{-1}$.
  \item For any $\xi\in(0,1]$, 
    $\forall t\ge \xi\delta$,\quad
    $|\delta \hat{C}_{[-M,M],\delta}(t)|\le
    c(\xi)=(1-\xi)+\hat{J}(\xi)/(\pi\xi)$.
  \end{enumerate}
\end{lem}
\noindent
Notice that $\hat{C}_{[-M,M],\delta}(t)=\delta^{-1}\hat{C}_{[-\delta
  M,\delta M],1}(\delta^{-1}t)$ and that
\begin{equation*}
  C_{[0,N],\delta}(t)=C_{[-N/2,N/2],\delta}(t-N/2).
\end{equation*}
\begin{proof}
  The main part of this lemma comes from Lemma~5.1
  of~\cite{Ramare*26-1} which is only a re-telling of part of the
  content of the paper~\cite{Vaaler*85} by J.\,D.~Vaaler.
  Lemma~5 of this latter paper implies in particular that
  \begin{equation*}
    2C_{[a,b],\delta}(x)\le
    2\cdot\1_{x\in[a,b]}+K(\delta(b-x))
    +K(\delta(x-a)).
  \end{equation*}
  The bound in~(2) is readily proved on noticing that $\hat{J}(\xi)$
  is even and non-increasing when $\xi\ge0$.
\end{proof}

\begin{figure}[!h]
  \includegraphics[scale=0.8]{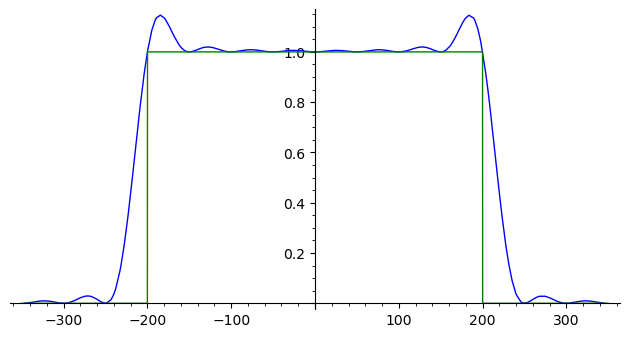}
  \caption{$C_{[-250,250], 1/50}(u)$}
\end{figure}

\section{An $L^2$-norm}

The present section is devoted to proving the next theorem.
\begin{thm}
  \label{myNorm}
  The function $C_{[-M,M],\delta}$ being taken from Lemma~\ref{BS}, we have
  \begin{equation*}
    \int_{-\infty}^{\infty}|C_{[-M,M],\delta}(t)|^2dt
    \le 2M+\myc \delta^{-1}
  \end{equation*}
  where
  \begin{equation}
    \label{defc}
    \myc = \frac43+
      \frac{2}{\pi^2}\int_{0}^1 \hat{J}(v)\hat{J}'(v)
    \frac{dv}{v}
    =0.898\,678\cdots
  \end{equation}
\end{thm}
A change of variable swiftly implies that
\begin{equation*}
  \int_{-\infty}^{\infty}|C_{[-M,M],\delta}(t)|^2dt
  =\delta^{-1}N(\lambda)
  \quad\text{where}\quad
  N(\lambda)=\int_{-\infty}^{\infty}|C_{[-\lambda,\lambda],1}(t)|^2dt
\end{equation*}
The direct plot displayed in Figure~\ref{Nlambda} seems to show
that $N(\lambda)-2\lambda$ is an increasing function
of~$\lambda$. Clarifying this question may lead to some interesting property.
\begin{figure}[!h]
  \centering
  \includegraphics[scale=0.8]{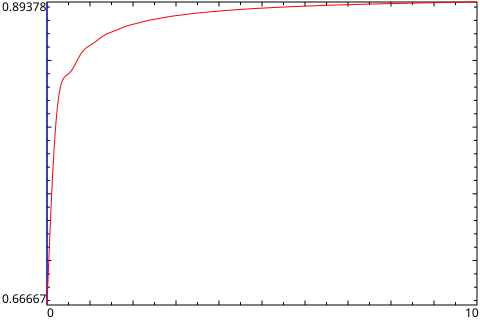}
  \caption{$\lambda\mapsto\int_\infty^{\infty}|C_{[-\lambda,\lambda],
      1}(t)|^2dt-2\lambda$}
  \label{Nlambda}
\end{figure}
This plot has been obtained by Pari/GP and numerical integration, on
setting \texttt{realprecision} to~300. Even like that, the reader
should be aware that numerical integration remains highly unreliable,
so that the above plot should be taken with a pinch of salt.
A detail of Figure~\ref{Nlambda} is displayed in
Figure~\ref{Nlambda-spe}. It shows an inflexion point around~$t=1/2$.
\begin{figure}[!h]
  \centering
  \includegraphics[scale=0.8]{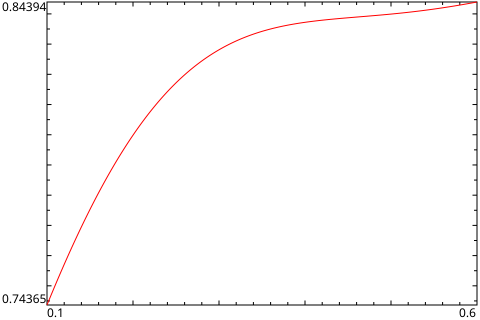}
  \caption{Detail of $\lambda\mapsto\int_\infty^{\infty}|C_{[-\lambda,\lambda],
      1}(t)|^2dt-2\lambda$}
  \label{Nlambda-spe}
\end{figure}

The proof unfolds in several lemmas.
By using the Parseval identity, we get
\begin{align*}
  N(\lambda)
  &=
    \int_0^1\Bigl((1-u)\cos(2\pi \lambda u)
    +\hat{J}(u)\frac{\sin(2\pi \lambda u)}{\pi u}\Bigr)^2du
  \\&=
  \int_0^1[(1-u)\cos(2\pi \lambda u)]^2du
  +\int_0^1(1-u)\hat{J}(u)\frac{\sin(4\pi \lambda u)}{\pi u}du
  \\&\qquad+\int_0^1\Bigl[\hat{J}(u)\frac{\sin(2\pi \lambda
  u)}{\pi u}\Bigr]^2du
  \\&=N_1(\lambda)+N_2(\lambda)+N_3(\lambda)
\end{align*}
say. Our aim is now to evaluate $N_1$, $N_2$ and $N_3$ separately.
It is worth recalling two 
classical functions:
\begin{equation}
  \label{eq:32}
  \sinc t=\frac{\sin t}{t}\quad
  \text{and}\quad
  \Si(t)=\int_0^t(\sinc v)dv.
\end{equation}
Let
us start with some properties on $\hat{J}$.
\begin{lem}
  \label{quantifyJhat}
  When $t\in[0,1]$ we have $\hat{J}(1-t)=1-\hat{J}(t)$.
  
  The function $\hat{J}'(t)$ is convex non-positive: decreasing over $[0,1/2]$, and
  increasing over $[1/2,1]$. We have $\hat{J}'(1-t)=\hat{J}'(t)$ and
  $\hat{J}'(0)=0$. The function $\hat{J}'(t)/t$ is negative increasing
  from $-2\pi^2/3$ to~0.
  
  The function $\hat{J}''(t)$ is increasing over $[0,1]$, with values
  $\hat{J}''(0)=-\hat{J}''(1)=-2\pi^2/3$. 
\end{lem}

\begin{proof}
  When $t\in[0,1]$, we find that
  \begin{align*}
    \hat{J}'(t)
    &=
      \pi (1-2t)\cot\pi t+1
      -\frac{\pi^2 t(1-t)}{\sin^2\pi t}
    \\&=
    \frac{\pi (1-2t)\sin 2\pi t
    -2\pi^2 t(1-t)}{2\sin^2\pi t}
    +1. 
  \end{align*}
  The first properties follows immediately from this expression.
  Let us examine the second derivative:
  \begin{align*}
    \hat{J}''(t)
    &=
      \frac{-2\pi \sin 2\pi t-2\pi^2 (1-2t)
      +2\pi^2 (1-2t)\cos 2\pi t}{2\sin^2\pi t}
    \\&\qquad
    -2\pi \cos(\pi t)\frac{\pi (1-2t)\sin 2\pi t
    -2\pi^2 t(1-t)}{2\sin^3\pi t}.
  \end{align*}
  A limited development and a plot are enough to justify the
  claimed properties.
  It is interesting to show how we perform these computations. The
  Pari/GP script is the next one.
\begin{verbatim}
{Jhat(u) =
   if(abs(u) > 1,
      return(0),
      return((1-abs(u))*cos(Pi*u)/sinc(Pi*u)+abs(u)));}

{Jhatprime(t) =
   return(1+(Pi*(1-2*t)*sin(2*Pi*t)-2*Pi^2*t*(1-t))/2/sin(Pi*t)^2);}

{Jhatprimedev(t) =
   return(1+(p*(1-2*t)*sin(2*p*t)-2*p^2*t*(1-t))/2/sin(p*t)^2);}
	  
{Jhatsecond(t) =
   my(aux = -2*Pi*sin(2*Pi*t)-2*Pi^2*(1-2*t), auxbis);
   aux += 2*Pi^2*(1-2*t)*cos(2*Pi*t);
   aux = aux /2/sin(Pi*t)^2;
   auxbis = -2*Pi*cos(Pi*t)*(Pi*(1-2*t)*sin(2*Pi*t)-2*Pi^2*t*(1-t));
   aux += auxbis /2/sin(Pi*t)^3;
   return(aux);}
   
{Jhatseconddev(t) =
   my(aux = -2*p*sin(2*p*t)-2*p^2*(1-2*t), auxbis);
   aux += 2*p^2*(1-2*t)*cos(2*p*t);
   aux = aux /2/sin(p*t)^2;
   auxbis = -2*p*cos(p*t)*(p*(1-2*t)*sin(2*p*t)-2*p^2*t*(1-t));
   aux += auxbis /2/sin(p*t)^3;
   return(aux);}
\end{verbatim}
  The limited development of $\hat{J}'(t)$ around $t=0$ is then simply
  obtained by the command \texttt{Jhatprimedev(t)+O(t\^{}4)} where the
  variable~$p$ stands for~$\pi$. We proceed similarly for $\hat{J}''(t)$.
\end{proof}
The case of $N_1$ is straightforward.
\begin{lem}
  \label{N1}
  We have $N_1(\lambda)=
  \frac16
  +\frac{1}{16\pi^2\lambda^2}
  \Bigl(1-\frac{\sin(4\pi \lambda)}{4\pi \lambda}\Bigr)$.
\end{lem}
Notice for later use that $1-\sinc t\le 1.217\cdots$ reached at
$t=4.493\dots$, the positive root of $\tan t=t$ in $[4, 3\pi/2)$.
\begin{proof}
We directly find that
\begin{align*}
  N_1(\lambda)
  &=
    \int_0^1(1-u)^2\frac{1+\cos(4\pi \lambda u)}{2}du
  =
  \frac16+\frac12\int_0^1(1-u)^2\cos(4\pi \lambda u)du
  \\&=
  \frac16+\frac{1}{4\pi\lambda}\int_0^1(1-u)\sin(4\pi \lambda u)du
  \\&=
  \frac16
  +\frac{1}{16\pi^2\lambda^2}
  -\frac{1}{16\pi^2\lambda^2}\int_0^1\cos(4\pi \lambda u)du
\end{align*}
and the claimed expression follows.
\end{proof}

\begin{figure}[!h]
  \centering
  \includegraphics[scale=0.8]{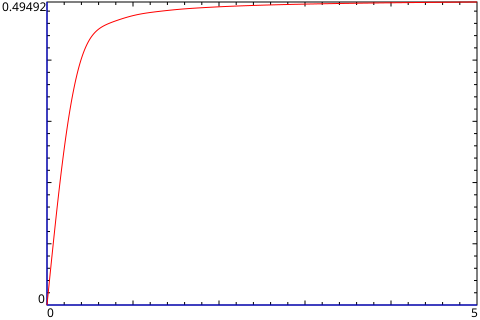}
  \caption{$\lambda\mapsto N_2(\lambda)$}
  \label{N2fig}
\end{figure}
\begin{lem}
  \label{N2}
  When $\lambda\le 1$, we have $N_2(\lambda)\le0.48$, while, when
  $\lambda\ge 0.84$, we have $N_2(\lambda)\le\frac12-\frac{0.518}{\lambda}$.
\end{lem}

\begin{proof}
  Let us set $\tau=4\pi\lambda$ and $g_2(x)=(1-x)\hat{J}(x)$ and write
\begin{align*}
  N_2(\lambda)
  &=
    \frac{1}{\pi}\int_0^\tau g_2(v/\tau)\frac{\sin v}{v}dv
    =
    \frac{-1}{\pi\tau}\int_0^\tau g_2(v/\tau)f_2'(v)dv
  \\&=
  \frac{1}{2}
  +\frac{1}{\pi\tau}\int_0^\tau g_2'(v/\tau)f_2(v)dv
\end{align*}
where
\begin{equation*}
  f_2(v)=\frac{\pi}{2}-\Si(v)
  =\int_v^{\infty}\frac{\sin u}{u}du
  =\frac{\cos v}{v}-\int_v^{\infty}\frac{\cos u}{u^2}du.
\end{equation*}
Let us further define
\begin{align*}
  F_2(w)&=\int_w^\infty f_2(v)dv
        =\int_w^\infty\frac{\cos v}{v}dv-\int_w^{\infty}(u-w)\frac{\cos u}{u^2}du
  \\&=\frac{-\sin w}{w}+\int_w^\infty\frac{\sin v}{v^2}dv
  -\int_w^{\infty}(\sin u)\frac{u-2w}{u^3}du
  \\&=\frac{-\sin w}{w}-2w\int_w^\infty(\sin u)\frac{du}{u^3}.
\end{align*}
This function may be computed via the script
\begin{verbatim}
F2(w)=-sin(w)/w-2*w*intnum(u=w,[oo,-I], sin(u)/u^3)
\end{verbatim}
It is negative increasing on~$[0,2.859\cdots]$. We have
$F_2(0)=-3$. It then oscillates between $0.285\cdots$ and
$-0.145\cdots$.
We numerically check that
\begin{equation}
  \label{BoundF2}
  \max_{w\ge 5/2}|F_2(w)|\le 3/10.
\end{equation}
It is classical that $|f_2(v)|\le\pi/2$.
Furthermore
$g_2'(t)=-\hat{J}(t)+(1-t)\hat{J}'(t)$ so that
$g_2'(0)=-1-\frac{2\pi^2}{3}$ and $g_2'(1)=0$. Therefore
\begin{align*}
  N_2(\lambda)
  &=
  \frac{1}{2}
  -\frac{1}{\pi\tau}\int_0^\tau g_2'(v/\tau)F_2'(v)dv
  \\&=
  \frac{1}{2}
    -\frac{1}{\pi\tau}[F_2(\tau)g_2'(1)-F_2(0)g_2'(0)]
    +\frac{1}{\pi\tau^2}\int_0^\tau g_2''(v/\tau)F_2(v)dv
  \\&=
  \frac{1}{2}
  -\biggl(1+\frac{2\pi^2}{3}\biggr)\frac{3}{4\pi^2\lambda}
  +\frac{1}{\pi\tau^2}\int_0^\tau g_2''(v/\tau)F_2(v)dv.
\end{align*}
The function $g''_2(t)=(1-t)\hat{J}''(t)-2\hat{J}(t)$ is concave. It
increases from 
$-2\pi^2/3$, reaches 0 at $t_0=0.239617\dots$, keeps increasing then
starts decreasing until $g''_2(1)=0$. 
Therefore, when $t_0\tau\ge 5/2$ by using~\eqref{BoundF2}, we get
\begin{align*}
  \frac{1}{\pi\tau^2}\int_0^\tau |g_2''(v/\tau)F_2(v)|dv
  &\le\frac{1}{4\pi^2\lambda}
  \biggl(3(g_2'(0)-g_2'(t_0))+\frac{3}{10}(g_2'(1)-g_2'(t_0)\biggr)
  \\&\le \frac{0.0575}{\lambda}.
\end{align*}
This proves the lemma when $\lambda\ge 0.83$. A numerical verification
closes the proof.
\end{proof}

\begin{lem}
  \label{N3}
  We have $N_3(\lambda)\le \lambda+\tfrac12\myc+\frac{0.236}{\lambda}$.
\end{lem}

\begin{proof}
  Let us set this time $\tau'=2\pi\lambda$ and
  $g_3(x)=\hat{J}(x)^2$. We find that 
\begin{equation*}
  N_3(\lambda)
  =\frac{2\lambda}{\pi}
    \int_0^{\tau'}g_3(v/\tau')
    \frac{\sin^2 v}{v^2}dv
  =-\frac{2\lambda}{\pi}
    \int_0^{\tau'}g_3(v/\tau')
    f'_3(v)dv
\end{equation*}
where
\begin{align}
  \label{eq:30}
  f_3(w)
  &=\int_w^\infty \frac{\sin^2v}{v^2}dv
  =\frac{1}{2w}-\int_w^\infty \frac{\cos(2v)}{2v^2}dv
  \\&=\frac{1}{2w}+\frac{\sin 2w}{4w^2}-\int_w^\infty
  \frac{\sin(2v)}{2v^3}dv
  =\frac{1}{2w}+f_4(w).
\end{align}
Therefore (the second line because $g'_3(w)\le 0$ when $w\ge0$)
\begin{align*}
  N_3(\lambda)
  &=
    \frac{2\lambda}{\pi} f_3(0)
    +\frac{2\lambda}{\pi\tau'}\int_0^{\tau'} g_3'(v/\tau')
    f_3(v)dv
  \\&\le
  \lambda
    +\frac{2\lambda}{\pi\tau'}\int_1^{\tau'} g_3'(v/\tau')
    f_3(v)dv
  \\&\le
  \lambda
  +\frac{2\lambda}{\pi\tau'}\int_1^{\tau'} g_3'(v/\tau')
    \frac{dv}{2v}
  +\frac{2\lambda}{\pi\tau'}\int_1^{\tau'} g_3'(v/\tau')
    f_4(v)dv
  \\&\le
  \lambda
  +\frac{1}{2\pi^2}\int_{1/\tau'}^1 g_3'(v)
    \frac{dv}{v}
  +\frac{2\lambda}{\pi\tau'}\int_1^{\tau'} g_3'(v/\tau')
    f_4(v)dv,
\end{align*}
then
\begin{align*}
  \int_1^{\tau'} g_3'(v/\tau')
    f_4(v)dv
  &=
  g'_3(1/\tau')F_4(1)-g'_3(1)F_4(\tau)
  +\frac{1}{\tau}\int_1^{\tau'} g_3''(v/\tau')
  F_4(v)dv
  \\&=
  g'_3(1/\tau')F_4(1)
  -\frac{1}{\tau'}\int_1^{\tau'} g_3''(v/\tau')
    F_4(v)dv
\end{align*}
where
\begin{align*}
  F_4(v)
  &=\int_v^\infty \frac{\sin 2w}{4w^2}dw
  -\int_v^\infty\int_w^\infty
  \frac{\sin(2u)}{2u^3}dudw
  \\&=
  \int_v^\infty \frac{\sin 2w}{4w^2}dw
  -\int_v^\infty
  \frac{(u-v)\sin(2u)}{2u^3}du
  \\&=
  -\int_v^\infty \frac{\sin 2w}{4w^2}dw
  +v\int_v^\infty
  \frac{\sin(2u)}{2u^3}du
\end{align*}
and therefore
\begin{align*}
  F_4(v)
  &=
  \frac{\cos 2v}{8v^2}-\int_v^\infty \frac{\cos 2w}{4w^3}dw
    +v
    \frac{\sin 2v}{16v^3}-\int_v^\infty \frac{3\sin 2w}{8w^4}dw
\end{align*}
We find that $F_4(1)=0.038\cdots\ge0$ by using

{\smallskip\hfill
  \texttt{intnum(v=1,[oo,-2*I],sin(2*v)*(1/4/v\^{}2-(v-1)/2/v\^{}3))}.\hfill
\smallskip}

\noindent
Furthermore $|F_4(v)|\le
(\tfrac18+\tfrac12+\tfrac{1}{16}+\tfrac{1}{8})/v^2=13/(16v^2)$.
We also find that $g_3''(t)=2\hat{J}(t)\hat{J}''(t)+2\hat{J}'(t)^2$
and, numerically, $|g_3''(t)|\le -g_3''(0)=4\pi^2/3$.
This leads to
\begin{equation*}
  \int_1^{\tau'} g_3'(v/\tau')
  f_4(v)dv
  \le \frac{1}{\tau'}\int_1^{\tau'} |g_3''(v/\tau')|
  |F_4(v)|dv
  \le \frac{4\pi^2/3}{2\pi\lambda}\int_1^{\infty} 
  \frac{13dw}{16w}\le \frac{9}{5\lambda}.
\end{equation*}
We thus find that
\begin{align*}
  N_3(\lambda)
  &\le
  \lambda
  +\frac{1}{2\pi^2}\int_{0}^1 g_3'(v)
    \frac{dv}{v}
   -\frac{1}{2\pi^2}\int_{0}^{1/\tau'} g_3'(v)
    \frac{dv}{v}
  +\frac{9}{5\pi^2\lambda}
  \\&\le
  \lambda
  +\frac{1}{2\pi^2}\int_{0}^1 g_3'(v)
    \frac{dv}{v}
   +\frac{2\pi^2/3}{2\pi^2\tau'}
  +\frac{9}{5\pi^2\lambda}
  \le\lambda+\myc+\frac{0.236}{\lambda}.
\end{align*}
This is what was to be poved.
\end{proof}

\begin{proof}[Proof of Theorem~\ref{myNorm}]
  It is enough to join Lemmas~\ref{N1}, \ref{N2} and~\ref{N3}.
\end{proof}




\section{Proof of Theorem~\ref{WeightedSieve}}

This section is devoted to the proof of
Theorem~\ref{WeightedSieve}. We present the quantities in some greater
generality for clarity, and show how the parameters are chosen when required.
We propose to  investigate the approximation
\begin{equation}
  \label{defJQ}
  J_Q(\alpha)=
  \frac{S(0)}{N}\sum_{q\le Q}\theta_q\sum_{\substack{1\le a\le q\\ (a,q)=1}}\hat{\psi}^*_q(a)
  \hat{C}_{[0,N],\delta_q}\biggl(\alpha-\frac{a}{q}\biggr)
\end{equation}
where ${C}_{[0,N],\delta_q}$ comes from Lemma~\ref{BS}, where the
weights $(\theta_q)$ are to be chosen (they are not assumed to be
non-negative) and where
we recall that
\begin{equation}
  \label{defdeltaq}
 \delta_q^{-1}=q(q+Q).
\end{equation}
The $\psi^*_q$ are defined in~\eqref{defpsistarq}. We expect $\hat{\psi}^*_q(a)
  \hat{C}_{[0,N],\delta_q}(\alpha-{a}/{q})$ to be a
  good approximation of $S(\alpha)$ for $\alpha$ close to $a/q$.

Notice that the function $J_Q(\alpha)$ takes its variable from $\mathbb{R}$ and not
from~$\mathbb{R}/\mathbb{Z}$ as is the case for~$S(\alpha)$. This is
for instance why we selected the variable $a$ by ``$1\le a\le
q,(a,q)=1$'' and not by ``$a\mode q$''. A moment's
thought will reveal to the reader that the support of $J_Q$ lies
inside~$[\delta_1,1+\delta_1]$, where we shall henceforth restrict our
argument.

\begin{lem}
  \label{normJQ}
  We have
  \begin{equation*}
    \int_{\delta_1}^{1+\delta_1} |J_Q(\alpha)|^2d\alpha\le
   \frac{|S(0)|^2}{N}\sum_{q\le Q}
    \theta_q^2g(q)(N+\myc  q(q+Q))
  \end{equation*}
\end{lem}

\begin{proof}
  We simply develop the square and obtain
  \begin{align*}
    \int_{\delta_1}^{1+\delta_1} |J_Q(\alpha)|^2d\alpha
    &=
      \frac{|S(0)|^2}{N^2}\sum_{q\le Q}
      \theta_q^2\sum_{\substack{1\le a\le q\\ (a,q)=1}}|\hat{\psi}^*_q(a)|^2
      \int_{\delta_1}^{1+\delta_1}
    \biggl|
    \hat{C}_{[0,N],\delta_q}\biggl(\alpha-\frac{a}{q}\biggr)
      \biggr|^2d\alpha
    \\&\le
    \frac{|S(0)|^2}{N^2}\sum_{q\le Q}\theta_q^2\sum_{a\mode
    q}|\hat{\psi}^*_q(a)|^2
    \int_{-\infty}^{\infty}
      \bigl|\hat{C}_{[0,N],\delta_q}(\beta)
    \bigr|^2d\beta
    \\&\le
    \frac{|S(0)|^2}{N^2}\sum_{q\le Q}
    \theta_q^2g(q)(N+\myc q(q+Q))
  \end{align*}
  by Theorem~\ref{myNorm}.
  The lemma is proved.
\end{proof}

\begin{lem}
  \label{scalSJQ}
  We have
   $\displaystyle \int_{\delta_1}^{1+\delta_1} S(\alpha)\overline{J_Q(\alpha)}d\alpha
    \ge
    \frac{|S(0)|^2}{N}
    \sum_{q\le Q}
    \theta_qg(q)$.
\end{lem}

\begin{proof}
  Set $\delta_q=1/(q(q+Q))$.
  We find that
  \begin{align*}
    \int_{\delta_1}^{1+\delta_1} S(\alpha)\overline{J_Q(\alpha)}d\alpha
    &=
      \frac{\overline{S(0)}}{N}\sum_{q\le Q}\theta_q
      \sum_{\substack{1\le a\le q\\ (a,q)=1}}\overline{\hat{\psi}^*_q(a)}
      \int_{\frac{a}{q}-\delta_q}^{\frac{a}{q}+\delta_q}
    S(\alpha)
    \overline{\hat{C}_{[0,N],\delta_q}\biggl(\frac{a}{q}-\alpha\biggr)}
      d\alpha
    \\&=
    \frac{\overline{S(0)}}{N}\sum_{q\le Q}\theta_q
    \sum_{\substack{1\le a\le q\\ (a,q)=1}}\overline{\hat{\psi}^*_q(a)}
    \int_{-\delta_q}^{\delta_q}
    S\biggl(\frac{a}{q}+\beta\biggr)
    \overline{\hat{C}_{[0,N],\delta_q}(\beta)}
      d\beta
    \\&-
      \frac{\overline{S(0)}}{N}\sum_{q\le Q}
    \theta_qg(q)\int_{-\infty}^{\infty}
    S(\beta)
    \overline{\hat{C}_{[0,N],\delta_q}(\beta)}
    d\beta
  \end{align*}
  the third last line following from Theorem~\ref{L2S0} applied to
  $(u_ne(n\beta))$ and the fact that 
  $\delta_q$ does not depend on~$a$.
  We appeal to Lemma~\ref{BS} to infer the lower bound
  \begin{equation*}
    \int_{\delta_1}^{1+\delta_1}
    S(\alpha)\overline{J_Q(\alpha)}d\alpha
    ´\ge
    \frac{S(0)^2}{N}\sum_{q\le Q}
    \theta_qg(q).
  \end{equation*}
    The lemma follows swiftly from there.
\end{proof}

\begin{proof}[Proof of Theorem~\ref{WeightedSieve}]
  For this proof, let us set
\begin{equation}
  \label{defGH}
  G^{(1)}(Q)=\sum_{q\le Q}\theta_q g(q).
\end{equation}
  We start from Lemma~\ref{scalSJQ} to which we apply the Cauchy
  inequality to infer that
  \begin{equation*}
    \frac{|S(0)|^4}{N^2}
    G^{(1)}(Q)^2
    \le
    \biggl|
    \int_{\delta_1}^{1+\delta_1}
    S(\alpha)\overline{J_Q(\alpha)}d\alpha
    \biggr|^2
    \le
    \|S\|_2^2
    \frac{|S(0)|^2}{N}\sum_{q\le Q}
    \theta_q^2g(q)(N+\myc  q(q+Q))
  \end{equation*}
  by Lemma~\ref{normJQ} and the Parseval equality.
  Consequently, we get
  \begin{equation*}
    |S(0)|^2\le N\|S\|_2^2/T
  \end{equation*}
  where
  \begin{equation*}
    T = G^{(1)}(Q)^2/\sum_{q\le Q}
    \theta_q^2g(q)(N+\myc  q(q+Q)).
  \end{equation*}
  In order to optimize the $\theta_q$, we notice that
  \begin{equation*}
    \biggl|
    \sum_{q\le Q}\theta_q g(q)
    \biggr|^2
    \le
    \sum_{q\le Q}\frac{g(q)}{N+\myc  q(q+Q)}
    \sum_{q\le Q}
    \theta_q^2g(q)(N+\myc  q(q+Q))
  \end{equation*}
  with equality when $\theta_q=1/(N+\myc  q(q+Q))$. This choice gives
  our theorem.
\end{proof}

\section{Two case studies}
\label{CaseStudies}
\subsection*{The Brun-Titchmarsh Inequality}
Let us start with the emblematic case of the Brun-Titchmarsh
Inequality. We restrict our attention to the case of primes in an
interval without any congruence condition. We set
\begin{equation}
  \label{eq:3}
  G(Q)=\sum_{q\le Q}\frac{\mu^2(q)}{\varphi(q)}
  =\log Q+c_0^*+\Ocal(1/\sqrt{Q}),
  \qquad
  c_0^*=1.332\,582\cdots
\end{equation}
see for instance~\cite{Ramare*18-9} with and
Eq.\,(2.11) in \cite{Rosser-Schoenfeld*62} by J.\,B.~Rosser \&{} L.~Schoenfeld.
Let us further define
\begin{equation}
  \label{deff}
  f(t)=\frac{1}{N+\myc t(t+Q)}-\frac{1}{N}
  =\frac{-\myc}{N}\frac{(Q+t)t}{N+\myc t(t+Q)}
  =\frac{-\myc}{N}g(t).
\end{equation}
The first derivative of $g$ is swiftly computed:
\begin{equation*}
  g'(t) =
  \frac{(2t+Q)N}{(N+\myc
  t(t+Q))^2}.
\end{equation*}
We therefore obtain
\begin{multline*}
  \frac{N}{\myc}\biggl(\frac{G(Q)}{N}-L\biggr)
  =
    \sum_{q\le Q}\frac{\mu^2(q)}{\varphi(q)}
    \biggl(-\int_q^Qg'(t)dt+g(Q)\biggr)
  =
  -\int_1^Q G(t) g'(t)dt + G(Q)g(Q)
  \\=
  -\int_1^Q (\log t+c_0^*) g'(t)dt + (\Log Q+c_0^*)g(Q)
  +\Ocal\biggl(\int_1^Q  \frac{-g'(t)dt}{\sqrt{t}} + \frac{g(Q)}{\sqrt{Q}}\biggr)
\end{multline*}
so that
\begin{equation*}
  \frac{N}{\myc}\biggl(\frac{G(Q)}{N}-L\biggr)
  =
  c_0^*g(1)+\int_1^Q \frac{g(t)dt}{t}
  +\Ocal\biggl(\frac{Q^2}{N\sqrt{Q}}\biggr)
\end{equation*}
from which we infer that
\begin{equation}
  \label{stepfour}
  \frac{N}{\myc}\biggl(\frac{G(Q)}{N}-L\biggr)
  =\int_1^Q \frac{g(t)dt}{t}
  +\Ocal\bigl(Q^{3/2}/N\bigr).
\end{equation}
Furthermore
\begin{align*}
  \int_1^Q \frac{g(t)dt}{t}
  &=
   \frac{1}{\myc}\int_{1/Q}^1 \frac{1+u}{\theta+ u(u+1)}du
  \\& =
  \frac{1}{\myc}\int_{0}^1 \frac{1+u}{\theta+ u(u+1)}du
  +\Ocal(1/(\theta Q))
\end{align*}
where $\theta=NQ^{-2}\myc^{-1}$.
The next lemma computes the integral.
\begin{lem}
  \label{Int1}
  When $\theta>1/4$, we find that
  \begin{equation*}
    \int_0^1 \frac{(u+1)du}{\theta+ u(1+u)}
    =
    \frac{1}{2}\log(1+2\theta^{-1})
    +\frac{1}{\sqrt{4\theta-1}}
    \biggl(\arctan\frac{3}{\sqrt{4\theta-1}}
    -\arctan\frac{1}{\sqrt{4\theta-1}}
    \biggr).
  \end{equation*}
\end{lem}
To detect possible calculation errors, this lemma has been checked numerically.
\begin{proof}
  Let us call $I(\theta)$ the integral to evaluate.
  We find that
  \begin{align*}
    I(\theta)
    &=
      \frac{1}{2}\int_0^1 \frac{2u+1}{\theta+ u(1+u)}du
      +\frac{1}{2}\int_0^1 \frac{du}{\theta+ u(1+u)}
    \\&=
      \frac{1}{2}\log(1+2\theta^{-1})
      +\frac{1}{2}\int_0^1 \frac{du}{\theta+ u(1+u)}.
  \end{align*}
  We now find, on assuming that $\theta>1/4$ and with
  $w=(u+1/2)/\sqrt{\theta-\frac14}$, that
  \begin{align*}
    \int_0^1\frac{du}{\theta+ u(1+u)}
    &
      =\int_0^1\frac{du}{(u+\frac12)^2+\theta-\frac14}
      =\frac{1}{\sqrt{\theta-\frac14}}
      \int_{1/(2\sqrt{\theta-\frac14})}^{3/(2\sqrt{\theta-\frac14})}
      \frac{dw}{w^2+1}
    \\&=
    \frac{2}{\sqrt{4\theta-1}}
    \biggl(\arctan\frac{3}{\sqrt{4\theta-1}}
    -\arctan\frac{1}{\sqrt{4\theta-1}}
    \biggr).
  \end{align*}
  The lemma is proved.
\end{proof}

Here is the final bound we reach when $\theta\ge1/4$:
\begin{multline*}
  NL=
  \frac12\log \frac{N}{\myc \theta}+c_0^*
  -
  \frac{1}{2}\log(1+2\theta^{-1})
 \\ -
  \frac{1}{\sqrt{4\theta-1}}
    \biggl(\arctan\frac{3}{\sqrt{4\theta-1}}
    -\arctan\frac{1}{\sqrt{4\theta-1}}
    \biggr)
  +\Ocal(N^{-1/4})
  \\\ge \frac12\log N+0.531
\end{multline*}
when $N$ is large enough, 
on choosing $\theta=3/2$.
It is less than the $5/6$
obtained in~\cite{Montgomery-Vaughan*73}.

\subsection*{The number of twin primes in an interval}
As noted in the introduction, Montgomery \& Vaughan's weighted large sieve is
particularly sharp to bound above the number of twin primes in an interval.
Let us apply our theorem on this case to prove Theorem~\ref{p+2}.
Let us define
\begin{equation}
  \label{eq:8}
  G(Q)=\sum_{q\le Q}\mu^2(q)\prod_{\substack{p|q\\ p>2}}\frac{2}{p-2}.
\end{equation}
This function is evaluated in Lemma~2
of~\cite{Riesel-Vaughan*83}, where it is proved that
\begin{multline}
  \label{eq:9}
  2\mathfrak{S}_2^{-1}G(Q)=(\log Q)^2+A_3\log Q+A_4+\Ocal(1/Q^{1/3}),
  \\\text{where $A_3=6.02347\cdots$ and
  $A_4=1.11407\cdots$.}
\end{multline}
On proceeding as in the Brun-Titchmarsh case, we reach the equivalent
of~\eqref{stepfour} but with the restriction $Q\gg\sqrt{N}$, namely
\begin{equation}
  \label{stepfourbis}
  \frac{2\mathfrak{S}_2^{-1}N}{\myc}\biggl(\frac{G(Q)}{N}-L\biggr)
  =\int_1^Q (2\log t+A_3)\frac{Q+t}{N+\myc t(t+Q)}dt
  +\Ocal\biggl({\frac{\log^2 Q}{Q^{1/3}}}\biggr).
\end{equation}
We immediately modify the above in
\begin{equation*}
  \frac{2\mathfrak{S}_2^{-1}N}{\myc}\biggl(\frac{G(Q)}{N}-L\biggr)
  =\frac{1}{\myc}
  \int_{0}^{1} (2\log u+2\log Q+A_3)\frac{1+u}{\theta+ u(u+1)}du
  +\Ocal\biggl({\frac{\log^2 Q}{Q^{1/3}}}\biggr).
\end{equation*}
where again $\theta=NQ^{-2}\myc^{-1}$. We simplify this in
\begin{align*}
  8\mathfrak{S}_2^{-1}NL
  &=\log^2 \frac{N}{\myc\theta}+2A_3\log \frac{N}{\myc \theta}
  -4\log \frac{N}{\myc\theta}
  \int_{0}^{1} \frac{1+u}{\theta+ u(u+1)}du
  +\Ocal(1)
  \\&=(\log N)^2 +2\biggl(A_3-\log(\myc\theta)
  -
  2\int_{0}^{1} \frac{1+u}{\theta+ u(u+1)}du\biggr)\log N+\Ocal(1).
\end{align*}

When $\theta>1/4$, Lemma~\ref{Int1} applies, leading to the constant
in front of the $\log N$ to be
\begin{equation*}
  2\biggl(A_3-\log(\myc\theta)
  -
  \log(1+2\theta^{-1})
    -\frac{2}{\sqrt{4\theta-1}}
    \biggl(\arctan\frac{3}{\sqrt{4\theta-1}}
    -\arctan\frac{1}{\sqrt{4\theta-1}}
  \biggr)\biggr).
\end{equation*}
The choice $\theta=1.19$ leads to a constant $\ge 8.866$.

\section{Proof of Theorem~\ref{better}, generic material}

We first develop the argument in full generality.
We propose to modify the proof of Theorem~\ref{WeightedSieve} and to
investigate the approximation 
\begin{equation}
  \label{defJQb}
  J_Q(\alpha)=
  \frac{S(0)}{N}\sum_{q\le Q}\theta_q\sum_{\substack{1\le a\le q\\ (a,q)=1}}\hat{\psi}^*_q(a)
  \1_{|\alpha-\frac{a}{q}|\le \delta_q}\hat{C}_{[0,N],\xi^{-1}_q\delta_q}\biggl(\alpha-\frac{a}{q}\biggr)
\end{equation}
where ${C}_{[0,N],\delta_q}$ comes from Lemma~\ref{BS}, where the
non-negative weights $(\theta_q)$ are to be chosen, as well as the
parameters $\xi_q\in(0,1]$ and where
we recall that
\begin{equation}
  \label{defdeltaq}
 \delta_q^{-1}= q(q+Q).
\end{equation}
The $\psi^*_q$ are defined in~\eqref{defpsistarq}. 


\begin{lem}
  \label{normJQb}
  We have
  \begin{equation*}
    \int_{\delta_1}^{1+\delta_1} |J_Q(\alpha)|^2d\alpha\le
   \frac{|S(0)|^2}{N}\sum_{q\le Q}
    \theta_q^2g(q)(N+\myc \xi_q \delta_q^{-1})
  \end{equation*}
\end{lem}

\begin{proof}
  We simply develop the square and obtain
  \begin{align*}
    \int_{\delta_1}^{1+\delta_1} |J_Q(\alpha)|^2d\alpha
    &=
      \frac{|S(0)|^2}{N^2}\sum_{q\le Q}
      \theta_q^2\sum_{\substack{1\le a\le q\\ (a,q)=1}}|\hat{\psi}^*_q(a)|^2
      \int_{\delta_1}^{1+\delta_1}
    \biggl|\1_{|\alpha-\frac{a}{q}|\le \delta_q}
    \hat{C}_{[0,N],\xi_q^{-1}\delta_q}\biggl(\alpha-\frac{a}{q}\biggr)
      \biggr|^2d\alpha
    \\&\le
    \frac{|S(0)|^2}{N^2}\sum_{q\le Q}\theta_q^2\sum_{a\mode
    q}|\hat{\psi}^*_q(a)|^2
    \int_{-\infty}^{\infty}
      \bigl|\hat{C}_{[0,N],\xi_q^{-1}\delta_q}(\beta)
    \bigr|^2d\beta
    \\&\le
    \frac{|S(0)|^2}{N^2}\sum_{q\le Q}
    \theta_q^2g(q)(N+\myc \xi_q q(q+Q))
  \end{align*}
  by Theorem~\ref{myNorm}.
  The lemma is proved.
\end{proof}

\begin{lem}
  \label{scalSJQb}
  For any even function $F$, we have
  \begin{multline*}
    \int_{\delta_1}^{1+\delta_1} S(\alpha)\overline{J_Q(\alpha)}d\alpha
    \ge
    \frac{|S(0)|^2}{N}
    \sum_{q\le Q}
    \theta_qg(q)
    -
    \frac{2\overline{S(0)}}{N}\sum_{q\le Q}
    \theta_qg(q)
    \int_{\delta_q}^{\xi_q^{-1}\delta_q}
    F(\beta)
    \overline{\hat{C}_{[0,N],\xi_q^{-1}\delta_q}(\beta)}
    d\beta.
    \\-
    \frac{2\overline{S(0)}}{N}\sum_{q\le Q}
    \theta_qg(q)
    \int_{\delta_q}^{\xi_q^{-1}\delta_q}
    |S(\beta)-F(\beta)|\frac{c(\xi_q\beta/\delta_q)}{\xi_q^{-1}\delta_q}
    d\beta.
  \end{multline*}
  The function $c$ is defined in Lemma~\ref{BS}.
\end{lem}
With $\xi_q=1$, this is Lemma~\ref{scalSJQ}.
\begin{proof}
  Set $\delta_q=1/(q(q+Q))$.
  We find that
  \begin{align*}
    \int_{\delta_1}^{1+\delta_1} S(\alpha)\overline{J_Q(\alpha)}d\alpha
    &=
      \frac{\overline{S(0)}}{N}\sum_{q\le Q}\theta_q
      \sum_{\substack{1\le a\le q\\ (a,q)=1}}\overline{\hat{\psi}^*_q(a)}
      \int_{\frac{a}{q}-\delta_q}^{\frac{a}{q}+\delta_q}
    S(\alpha)
    \1_{|\alpha-\frac{a}{q}|\le \delta_q}
    \overline{\hat{C}_{[0,N],\xi_q^{-1}\delta_q}\biggl(\frac{a}{q}-\alpha\biggr)}
      d\alpha
    \\&=
    \frac{\overline{S(0)}}{N}\sum_{q\le Q}\theta_q
    \sum_{\substack{1\le a\le q\\ (a,q)=1}}\overline{\hat{\psi}^*_q(a)}
    \int_{-\delta_q}^{\delta_q}
    S\biggl(\frac{a}{q}+\beta\biggr)
    \overline{\hat{C}_{[0,N],\xi_q^{-1}\delta_q}(\beta)}
      d\beta
    \\&=
      \frac{\overline{S(0)}}{N}\sum_{q\le Q}
    \theta_qg(q)\int_{-\infty}^{\infty}
    S(\beta)
    \overline{\hat{C}_{[0,N],\xi_q^{-1}\delta_q}(\beta)}
    d\beta
    \\&\qquad-
    \frac{\overline{S(0)}}{N}\sum_{q\le Q}
    \theta_qg(q)\int_{|\beta|\ge\delta_q}
    S(\beta)
    \overline{\hat{C}_{[0,N],\xi_q^{-1}\delta_q}(\beta)}
    d\beta
  \end{align*}
  the third last line following from Theorem~\ref{L2S0} applied to
  $(u_ne(n\beta))$ and the fact that 
  $\delta_q$ does not depend on~$a$.
  Lemma~\ref{BS} gives us the bound
  \begin{equation*}
    \big|\xi_q^{-1}\delta_q\hat{C}_{[0,N],\xi_q^{-1}\delta_q}(\beta)\bigr|
    \le c(\xi_q)
  \end{equation*}
  when $|\beta|\ge\delta_q$, but it also implies
  \begin{equation}
    \label{mb}
    \big|\xi_q^{-1}\delta_q\hat{C}_{[0,N],\xi_q^{-1}\delta_q}(\beta)\bigr|
    \le c(\xi_q|\beta|/\delta_q).
  \end{equation}
  The lemma follows swiftly from there.
\end{proof}

\section{Auxiliaries}

Let us first recall (a simpler version of)
\cite[Theorem~13.1]{Ramare*22-0}.
\begin{lem}
  \label{order}
  Let $Q\ge2$ be a fixed parameter. Assume that $f$ is a non-negative
  multiplicative function such that
  \begin{equation*}
    \forall D\in[1,Q],\quad\sum_{\substack{p\ge 2,\nu\ge1\\ p^\nu\le
        D}}
    f(p^\nu)\log(p^\nu)\le KD
  \end{equation*}
  for some constant $K\ge0$. Then we have
  \begin{equation*}
    \sum_{d\le Q}f(d)\le \frac{(K+1)Q}{\log Q}\sum_{d\le Q}f(d)/d.
  \end{equation*}
\end{lem}

\begin{lem}
  \label{orderp}
  Under the hypotheses of Lemma~\ref{order}, we have
  \begin{equation*}
    \sum_{d\le Q}\frac{f(d)}{\sqrt{d}}
    \ll_K \frac{\sqrt{Q}}{\log Q}\sum_{d\le Q}f(d)/d.
  \end{equation*}
\end{lem}
\begin{proof}
  We readily find that
  \begin{align*}
    \sum_{d\le Q}\frac{f(d)}{\sqrt{d}}
    &=
    \sum_{d\le Q}f(d)\biggl(\int_{d}^Q\frac{dt}{2t^{3/2}}+\frac{1}{\sqrt{Q}}\biggr)
    \\&\le
    \int_1^Q  \sum_{d\le t}f(d)
    \frac{dt}{2t^{3/2}}+\frac{1}{\sqrt{Q}}
    \sum_{d\le Q}f(d)\ll \frac{\sqrt{Q}}{\log Q}
    \sum_{d\le Q}f(d)/d
  \end{align*}
  by Lemma~\ref{order} and since $\sum_{d\le t}f(d)/d\le \sum_{d\le
    D}f(d)/d$. The lemma follows readily.
\end{proof}

\begin{lem}\label{density}
Let $h$ be a non-negative multiplicative function.
Let $\kappa$, $L$ and $A$ be three non-negative real parameters
such that
\begin{equation*}
  \left\{
    \begin{array}{l}
      \displaystyle
      \sum_{\substack{ p\ge2, \nu\ge1\\ p^{\nu}\le Q}}
      h\bigl(p^{\nu}\bigr)\Log\bigl(p^{\nu}\bigr)=
      \kappa\Log Q+\Oc^*(L)\qquad(Q\ge1),
      \\\displaystyle
      \sum_{p\ge2}
      \sum_{\substack{\nu,k\ge1}}
      h\bigl(p^k\bigr)h\bigl(p^{\nu}\bigr)\Log\bigl(p^{\nu}\bigr)
      \le A.
    \end{array}
  \right.
\end{equation*}
Then, when $D\ge\exp(2(L+A))$, we have
$$
\sum_{d\le D}h(d)= C\left(\Log D\right)^{\kappa}
\left(1+\Oc^*(B/\Log D)\right)
$$
with $B=\displaystyle2(L+A)\bigl(1+2(\kappa+1)e^{\kappa+1}\bigr)$ and
\begin{equation}
  C=\displaystyle\frac{1}{\Gamma(\kappa+1)}
     \prod_{p\ge2}\biggl\{
     \biggl(1-\frac1p\biggr)^{\kappa}
      \sum_{\nu\ge0}h\bigl(p^{\nu}\bigr)
      \biggr\}.\label{defC}
\end{equation}
\end{lem}
It is not difficult by following Wirsing in~\cite{Wirsing*61}

\begin{lem}\label{density+}
  Let $f$ be a non-negative multiplicative function
and  $\kappa$ be non-negative real parameter
such that
\begin{equation*}
  \left\{
    \begin{array}{l}
      \displaystyle
      \sum_{\substack{ p\ge2, \nu\ge1\\ p^{\nu}\le Q}}
      f\bigl(p^{\nu}\bigr)\Log\bigl(p^{\nu}\bigr)=
      \kappa Q+\Oc(Q/\Log(2Q))\qquad(Q\ge1),
      \\\displaystyle
      \sum_{p\ge2}
      \sum_{\substack{\nu,k\ge 1\\ p^{\nu+k}\le Q}}
      f\bigl(p^k\bigr)f\bigl(p^{\nu}\bigr)\Log\bigl(p^{\nu}\bigr)
      \ll \sqrt{Q},
    \end{array}
  \right.
\end{equation*}
then, for any $C^1$-function $F$ with a bounded derivative and any $\eta\in[0,1]$, we have
\begin{multline*}
    \sum_{\eta Q\le q\le Q}f(q)F(q/Q)
  =
    \kappa C \cdot Q(\log Q)^{\kappa-1}\int_\eta^1 F(t)dt
    \\+\Ocal\biggl(\int_\eta^1 t(1+|\log t|)|F'(t)|dt
    (\log Q)^{\kappa-2}Q\biggr)
\end{multline*}
where~$C$ is as in Lemma~\ref{density}.
\end{lem}
In essence, this lemma is rather usual but we need the precision of
the error term in our two instances of application as we may not
assume that $F'(t)$ is bounded around~0. 
\begin{proof}
  In general, let us set
  $G(D)=\sum_{ q\le D}f(q)$. 
  On following the proof of
  \cite[Theorem 13.4]{Ramare*22-0}, the reader will readily see that
  the error term stated there as $o(1)$ is in fact
  $\Ocal(1/\log(D+2))$, i.e.
  \begin{equation}
    \label{eq:18}
    G(D)=C\kappa D(\log D)^{\kappa-1}(1+\Ocal(1/\log(D+2)).
  \end{equation}
  We may always assume that $\eta\ge1/Q$.
  We find that
\begin{align*}
  \sum_{\eta Q \le q\le Q}f(q)F(q/Q)
  &=
  -\sum_{\eta Q\le q\le Q}f(q)\int_{q/Q}^1 F'(t)dt
  +F(1)(G(Q)-G(\eta Q))
  \\&=-\int_\eta^1 (G(tQ)-G(\eta Q)) F'(t)dt
  +F(1)(G(Q)-G(\eta Q))
  \\&=
  -\int_\eta^1 G(tQ) F'(t)dt
  +F(1)G(Q)-F(\eta)G(\eta Q).
\end{align*}
We may continue as follows:
\begin{equation*}
  -\int_\eta^1 G(tQ) F'(t)dt
  =
    -\int_\eta^1 C\kappa(\log tQ)^{\kappa-1}tQ F'(t)dt
    +\int_\eta^1 \Ocal\bigl((\log (tQ+2))^{\kappa-2}\bigr)tQ F'(t)dt.
  \end{equation*}
  Since $t\ge \eta\ge 1/Q$, we find that
  \begin{equation*}
    (\log tQ)^{\kappa-1}
    =(\log Q)^{\kappa-1} +\Ocal(|\log t|(\log Q)^{\kappa-2}).
  \end{equation*}
  Consequently, when $\kappa\ge2$, we find that
  \begin{multline*}
    -\int_\eta^1 G(tQ) F'(t)dt
    =-\int_\eta^1 C\kappa(\log Q)^{\kappa-1}tQ F'(t)dt
    \\+\Ocal\biggl(Q(\log Q)^{\kappa-2}\int_\eta^1t|F'(t)(1+|\log t|)|dt\biggr).
  \end{multline*}
  The reader will readily close the proof from there in this case. Let
  us next consider the case $\kappa\in[1,2]$. When $\eta Q\ge
  \sqrt{Q}$, we may use $(\log (2+tQ))^{\kappa-2}\ll(\log
  Q)^{\kappa-2}$ and the proof gets readily closed. Otherwise, the
  part from $\eta$ to $1/\sqrt{Q}$ is easily shown to be negligible. 
  The proof of lemma, or more precisely our sketch of it, is now complete.
\end{proof}

\section{Proof of Theorem~\ref{better} and initiating the proof of Theorem~\ref{bettere}}
\label{sketch}

Let us first recall Corollary~2.6 of \cite{Ramare*26-1}.
\begin{lem}[Regularity of large sifted sets]
  \label{ManypsL20}
  Let $Z$ be the number of points $n$ that belong to every $\Kcal_q$,
  where $(\Kcal_q)_{q\le Q}$ is
  a consistent multiplicative sequence
  that satisfies the Johnsen-Gallagher
  condition. Assume that $L(Q)$ satisfies
  $L(Q)=C (\log Q)^\kappa(1+\Ocal(1/\log Q))$ when $Q\ge2$, for some
  $C>0$ and $\kappa\ge0$.
  For any parameter $A\ge1$, we have 
  \begin{equation*}
    \int_{-(\log N)^A/N}^{(\log N)^A/N}\biggl|S(\beta)-\frac{Z}{N}
    \sum_{n\le N} e(n\beta)\biggr|^2d\beta
    \ll_A
    \biggl(\frac{N}{L(\sqrt{N})}-Z
    \biggr)\frac{1}{(\log N)^{\kappa}}+
    \frac{N\log\log N}{(\log N)^{2\kappa+1}}
  \end{equation*}
  where $S(\beta)=\sum_n e(n\beta)$.
\end{lem}
Here is the stronger version of Lemma~\ref{scalSJQ} that we will use.
\begin{lem}
  \label{Out2}
  Let $A\ge \frac52+\kappa$ be a constant. Let us assume that $Q\le
  \sqrt{N}$ and that
  \begin{equation}
    \tag{Hyp}
    S(0)=Z\ge \frac{N}{L(\sqrt{N}(\log N)^A)}.
  \end{equation}
  When $0\le \theta_q\le 1$, $Q\le \pi\sqrt{N}/2$ and
  $\xi_q=\min(1,N/[q(q+Q) (\log N)^A])$, we have 
  \begin{equation*}
    \int_{\delta_1}^{1+\delta_1} S(\alpha)\overline{J_Q(\alpha)}d\alpha
    \ge
    \frac{Z^2}{N}
    \sum_{q\le Q}
    \theta_qg(q)\biggl(1-
    \1_{q\ge Q_0}
    \frac{2q(q+Q)}{\pi^2N}\biggr)
    +
    \Ocal_{A}\biggl(
    \frac{\sqrt{\log\log N}}{(\log N)^{3/2}}Z
    \biggr)
  \end{equation*}
  where $Q_0=N/(2Q(\log N)^A)$.
\end{lem}

\begin{proof}
  We start with Lemma~\ref{scalSJQb} in which we plug the estimate of
  Lemma~\ref{ManypsL20}. Let us first notice that
  \begin{align*}
    \int_{\delta_q}^{\xi_q^{-1}\delta_q}
    \biggl(\frac{c(\xi_q\beta/\delta_q)}{\xi_q^{-1}\delta_q}\biggr)^2
    d\beta
    &=
      \int_{\xi_q}^{1}
    \frac{\xi_qc(t)^2}{\delta_q}
      dt
      \le
      \frac{\xi_q}{\delta_q}\int_{\xi_q}^{1}
    (1-t+1/(\pi t))^2
      dt
    \\&\le
    \frac{1}{\pi^2\delta_q}(1+\Ocal(\xi_q\log\xi_q))\ll1/\delta_q.
    \end{align*}
    Our second preparatory estimate is the next one  with $Z=S(0)$.
  \begin{align*}
    \int_{\delta_q}^{\xi_q^{-1}\delta_q}
    \biggl|  \frac{Z}{N}
    \int_{0}^{N} e(\beta t)dt
    \biggr|
    \frac{c(\xi_q\beta/\delta_q)}{\xi_q^{-1}\delta_q}
    d\beta
    &=
    \frac{Z}{N}
    \int_{\delta_q}^{\xi_q^{-1}\delta_q}\frac{|\sin\pi\beta
      N|}{\pi\beta}
      \frac{c(\xi_q\beta/\delta_q)}{\xi_q^{-1}\delta_q}d\beta
    \\&=
    \frac{Z}{N}
    \int_{\xi_q}^{1}\frac{|\sin\pi t
      \delta_q N/\xi_q|}{\pi t}
      \frac{c(t)}{\xi_q^{-1}\delta_q}dt
     \\&\le
     \frac{Z}{N}
       \int_{\xi_q}^{1}\frac{1}{\pi t}
    \frac{1-t+\frac{1}{\pi t}}{\xi_q^{-1}\delta_q}dt
    \le \frac{Z(1+\Ocal(\xi_q\log\xi_q))}{\pi^2\delta_q N}.
  \end{align*}
  We may now use the inequality established in Lemma~\ref{scalSJQb}
  together with~\eqref{mb}. We obtain:
  \begin{multline}
    \label{innerstep}
    \int_{\delta_1}^{1+\delta_1} S(\alpha)\overline{J_Q(\alpha)}d\alpha
    \ge
    \frac{Z^2}{N}
    \sum_{q\le Q}
    \theta_qg(q)
    \\-
    \frac{2Z}{N}\sum_{q\le Q}
    \theta_qg(q)
    \int_{\delta_q}^{\xi_q^{-1}\delta_q}
    \biggl|S(\beta)-\frac{Z}{N}\sum_{n\le N}e(n\beta)\biggr|
    \frac{c(\xi_q\beta/\delta_q)}{\xi_q^{-1}\delta_q}d\beta
    \\-
    \frac{2Z}{N}\sum_{q\le Q}
    \theta_qg(q)
    \int_{\delta_q}^{\xi_q^{-1}\delta_q}
    \biggl|\frac{Z}{N}\sum_{n\le N}e(n\beta)\biggr|
    \frac{c(\xi_q\beta/\delta_q)}{\xi_q^{-1}\delta_q}d\beta.
  \end{multline}
  By using the Cauchy inequality and on plugging there the estimate of
  Lemma~\ref{ManypsL20}, we find that 
  \begin{equation*}
    \Biggl|\int_{\delta_q}^{\xi_q^{-1}\delta_q}
    \biggl|S(\beta)-\frac{Z}{N}\sum_{n\le N}e(n\beta)\biggr|
    \frac{c(\xi_q\beta/\delta_q)}{\xi_q^{-1}\delta_q}d\beta\Biggr|^2
    \ll
    \biggl(\frac{N}{L(\sqrt{N})}-Z
    \biggr)\frac{1}{\delta_q(\log N)^{\kappa}}+
    \frac{N\log\log N}{\delta_q(\log N)^{2\kappa+1}}
  \end{equation*}
  since $\xi_q^{-1}\delta_q= (\log N)^A/N$ when $q\ge Q_1$ and
  $\xi_q=1$ otherwise (for some $Q_1$ that can be explicitely determined).
  Let us assume recall Hypothesis~$(Hyp)$ and set
  $T=A(\log N)^{\kappa-1}\log\log N$.
  We get
  \begin{equation*}
    \frac{N}{L(\sqrt{N})}-Z\le
    \frac{NT}{L(\sqrt{N})^2}\ll_A \frac{NT}{(\log N)^{2\kappa}}
    \ll_A \frac{ZT}{(\log N)^\kappa}.
  \end{equation*}
  Consequently, we find that, when $q\ge Q_1$,
  \begin{align*}
    \Biggl|\int_{\delta_q}^{\xi_q^{-1}\delta_q}
    \biggl|S(\beta)-\frac{Z}{N}\sum_{n\le N}e(n\beta)\biggr|
    \frac{c(\xi_q\beta/\delta_q)}{\xi_q^{-1}\delta_q}d\beta\Biggr|^2
    &\ll
    \frac{Z}{\delta_q(\log N)^{\kappa}}
    \biggl(\frac{T}{(\log N)^{\kappa}}+
    \frac{\log\log N}{\log N}\biggr)
    \\&\ll
    \frac{Z\log\log N}{\delta_q(\log N)^{1+\kappa}}.
  \end{align*}
  We have thus reached (on using $|\theta_q| \ll 1$)
  \begin{multline*}
    \int_{\delta_1}^{1+\delta_1} S(\alpha)\overline{J_Q(\alpha)}d\alpha
    \ge
    \frac{Z^2}{N}
    \sum_{q\le Q}
    \theta_qg(q)
    -
    \Ocal(1)\frac{2Z}{N}\sum_{Q_1\le q\le Q}g(q)
    \sqrt{\frac{Z\log\log N}{\delta_q(\log N)^{1+\kappa}}}
    \\-
    \frac{2Z^2}{\pi^2N^2}\sum_{Q_1\le q\le Q}
    \frac{\theta(q)g(q)}{\delta_q}
    -\Ocal(1)\frac{Z^2}{\pi^2N^2}\sum_{Q_1\le q\le Q}
    \frac{g(q)\xi_q(\log\xi_q)}{\delta_q}.
  \end{multline*}
  This simplifies in
  \begin{multline*}
    \int_{\delta_1}^{1+\delta_1} S(\alpha)\overline{J_Q(\alpha)}d\alpha
    \ge
    \frac{Z^2}{N}
    \sum_{q\le Q}
    \theta_qg(q)
    -
    \Ocal\biggl(\frac{Z}{N}\sum_{q\le Q}g(q)\sqrt{q}
    \sqrt{\frac{QZ\log\log N}{(\log N)^{1+\kappa}}}\biggr)
    \\-
    \frac{2Z^2}{\pi^2N^2}\sum_{Q_1\le q\le Q}
    \theta(q)g(q)q(q+Q)
    -\Ocal\biggl(\frac{Z^2}{N(\log N)^{A-1-\kappa}}\biggr).
  \end{multline*}
  Lemma~\ref{orderp} applied to $f(d)=dg(d)$ gives us the next bound:
  \begin{equation*}
    \sum_{q\le Q}g(q)\sqrt{q}
    \ll \sqrt{Q}(\log Q)^{\kappa-1}
  \end{equation*}
  and therefore
   \begin{multline*}
    \int_{\delta_1}^{1+\delta_1} S(\alpha)\overline{J_Q(\alpha)}d\alpha
    \ge
    \frac{Z^2}{N}
    \sum_{q\le Q}
    \theta_qg(q)
    -
    \frac{2Z^2}{\pi^2N^2}\sum_{q\le Q}
    \theta(q)g(q)q(q+Q)
    \\-
    \Ocal\biggl(\frac{Z}{N}\sqrt{Q}(\log N)^{\kappa-1}
    \sqrt{\frac{QZ\log\log N}{(\log N)^{1+\kappa}}}\biggr)
    -\Ocal\biggl(\frac{Z^2}{N(\log N)^{A-1-\kappa}}\biggr).
  \end{multline*}
  The error term reads
  \begin{equation*}
    \Ocal\biggl(
    \biggl(\frac{Q}{\sqrt{Z}(\log N)^{\frac{\kappa+3}{2}}}
    \sqrt{\log\log N}
    +\frac{1}{(\log N)^{A-1-2\kappa}}\biggr)\frac{Z^2}{N}(\log N)^\kappa
    \biggr)
  \end{equation*}
  or also
  \begin{equation*}
    \Ocal\biggl(
    \biggl(\frac{Q\sqrt{\log\log N}}{\sqrt{N}(\log N)^{\frac{3}{2}}}
    +\frac{1}{(\log N)^{A-1-2\kappa}}\biggr)Z
    \biggr).
  \end{equation*}
  The condition $q\ge Q_1$ is to ensure that $q(q+Q)(\log N)^A\ge N$,
  and can be degraded in $2q Q(\log N)^A\ge N$. 
  The lemma follows readily.
\end{proof}

\begin{proof}[Proof of Theorem~\ref{better}]
  By using Cauchy's inequality on
  $\int_{\delta_1}^{1+\delta_1}
  S(\alpha)\overline{J_Q(\alpha)}d\alpha$, and on using
  Lemmas~\ref{normJQb} and~\ref{Out2}, we find that
  \begin{multline*}
    \biggl|
    \frac{Z^2}{N}
    \sum_{q\le Q}
    \theta_qg(q)
    \biggl(1-
    \1_{q\ge Q_0}
    \frac{2 q(q+Q)}{\pi^2 N}\biggr)
    -\Ocal\biggl(
    \frac{\sqrt{\log\log N}}{(\log N)^{3/2}}Z\biggr)
    \biggr|^2
    \\\le
    Z\biggl(\frac{Z^2}{N}\sum_{q\le Q}
    \theta_q^2g(q)
    +\Ocal\biggl(\frac{Z^2Q^2}{N(\log N)^3}\biggr)\biggr).
  \end{multline*}
  Since $Q^2\ll N$, this simplifies in
  \begin{multline*}
    Z
    \biggl|
    \sum_{q\le Q}
    \theta_qg(q)
    \biggl(1-
    \1_{q\ge Q_0}
    \frac{2 q(q+Q)}{\pi^2 N}\biggr)
    -\Ocal\biggl(
    \frac{\sqrt{\log\log N}}{(\log N)^{3/2}}(\log N)^{\kappa}\biggr)
    \biggr|^2
    \\\le
    N\bigl(1+\Ocal((\log N)^{-2})\bigr)\sum_{q\le Q}
    \theta_q^2g(q).
  \end{multline*}
  This further simplifies into
  \begin{multline*}
    Z\biggl(1+\Ocal\biggl(\frac{\sqrt{\log\log N}}{(\log N)^{3/2}}\biggr)\biggr)
    \biggl|
    \sum_{q\le Q}
    \theta_qg(q)
    \biggl(1-
    \1_{q\ge Q_0}
    \frac{2 q(q+Q)}{\pi^2 N}\biggr)
    \biggr|^2
    \\\le
    N\bigl(1+\Ocal((\log N)^{-2})\bigr)\sum_{q\le Q}
    \theta_q^2g(q)
  \end{multline*}
  i.e.
  \begin{equation*}
    Z
    \biggl|
    \sum_{q\le Q}
    \theta_qg(q)
    \biggl(1-
    \1_{q\ge Q_0}
    \frac{2 q(q+Q)}{\pi^2  N}\biggr)
    \biggr|^2
    \le
    N\biggl(1+\Ocal\biggl(\frac{\sqrt{\log\log N}}{(\log
      N)^{3/2}}\biggr)\biggr)
    \sum_{q\le Q}
    \theta_q^2g(q).
  \end{equation*}
  We select
  \begin{equation}
    \theta_q=1-
    \frac{2 q(q+Q)}{\pi^2 N}
  \end{equation}
  and $Q\le \pi\sqrt{N}/2$, so that this quantity remains non-negative.
  We have reached
  \begin{equation*}
    Z
    \sum_{q\le Q}
    g(q)
    \biggl(1-
    \1_{q\ge Q_0}
    \frac{2 q(q+Q)}{\pi^2 N}\biggr)^2
    \le
    N\biggl(1+\Ocal\biggl(\frac{\sqrt{\log\log N}}{(\log
      N)^{3/2}}\biggr)\biggr).
  \end{equation*}
  Forgetting the condition $q\ge Q_0$ ends the proof of our theorem.
\end{proof}
Let us now turn towards the proof of a more explicit version of
Theorem~\ref{better}, i.e. Theorem~\ref{bettere}. We first prove a
simplified version before in the next two sections in a numerically
more refined proof.
\begin{proof}[Proof of a simplified version of Theorem~\ref{bettere}]
We now find that, with $Y=2/(\pi^2N)$,
  \begin{equation*}
    \biggl(1-
    \frac{2 q(q+Q)}{\pi^2 N}\biggr)^2
    =
    Y^2q^4 + 2Y^2Qq^3 + (Y^2Q^2 - 2Y)q^2 - 2YQq + 1.
  \end{equation*}
  Let us set
  \begin{equation}
    \label{eq:14}
    H=\sum_{q\le Q}
    g(q)
    \biggl(1-
    \frac{2 q(q+Q)}{\pi^2 N}\biggr)^2.
  \end{equation}
  On using Lemma~\ref{density+}, we find that
  \begin{align*}
    H
    &=
    \sum_{q\le Q}
    g(q)
    \bigl(Y^2q^4 + 2Y^2Qq^3 + (Y^2Q^2 - 2Y)q^2 - 2YQq + 1\bigr)
    \\&=
    \biggl(Y^2\frac{\kappa C}{4}Q^4 + 2Y^2Q\frac{\kappa C}{3}Q^3 +
    Y(YQ^2 - 2)\frac{\kappa C}{2}Q^2 - 2YQ\kappa C
    \\&\qquad+
    C \log Q+C\cdot c_0^*+o(1)\biggr)(\log Q)^{\kappa-1}.
  \end{align*}
  We set $x=Q/\sqrt{N}$ and get
  \begin{multline*}
    \frac{H}{C(\log Q)^{\kappa-1}}-\frac{\log N}2
    \\=
    \frac{4\kappa}{\pi^4}\biggl(\frac{1}{4} +
    \frac{2}{3}
    +\frac{1}{2}
    \biggr)
    x^4 -
    \frac{4\kappa}{\pi^2}\biggl(\frac{1}{2}+1\biggr)x^2
    +
    \log x+c_0^*+o(1),
  \end{multline*}
  i.e.
  \begin{equation}
    \label{getH}
    \frac{H}{C(\log Q)^{\kappa-1}}-\frac{\log N}2
    =
    r(x)=\frac{17\kappa}{3\pi^4} x^4 -
    \frac{6\kappa}{\pi^2}x^2
    +
    \log x+c_0+o(1).
  \end{equation}
  Remember also that $\theta_q$ had to be non-negative, so that we
  have
  \begin{equation*}
    1-\frac{4Q^2}{\pi^2N}=1-4(x/\pi)^2\ge0.
  \end{equation*}
  Notice that
  \begin{equation*}
    (\log Q)^{\kappa-1}=(\tfrac12\log N)^{\kappa-1}
    +(\kappa-1)(\log x)(\tfrac12\log N)^{\kappa-2}(1+o(1))
  \end{equation*}
  so that
  \begin{align*}
    H/C
    &=(\tfrac12\log N)^{\kappa}
    +(r(x)+(\kappa-1)\log x)(\tfrac12\log N)^{\kappa-1}(1+o(1)).
    \\&=(\tfrac12\log N)^{\kappa}
    +\biggl(\frac{17\kappa}{3\pi^4} x^4 -
    \frac{6\kappa}{\pi^2}x^2
    +
    \kappa \log x+c_0\biggr)(\tfrac12\log N)^{\kappa-1}(1+o(1)).
  \end{align*}
  The derivative in $x$ of the constant term  reads
  \begin{equation*}
    \frac{68}{3\pi^4} x^3 -
    \frac{12}{\pi^2}x
    +
    \frac1x
  \end{equation*}
  so that it vanishes when
  \begin{equation*}
    \frac{68}{3} (x/\pi)^4 -
    12 (x/\pi)^2
    +
    1=0.
  \end{equation*}
  Let us set
  \begin{equation*}
    (x_{\pm}/\pi)^2=\frac{6\pm\sqrt{36-68/3}}{68/3}.
  \end{equation*}
  The above function of~$x$ increases until $x_-$, then decreases
  until $x_+$ and increases after this value.
  We find that
  \begin{equation*}
    (x_-/\pi)^2=0.103\,611\cdots
    \le 1/4\le
    0.425\,800\cdots=(x_+/\pi)^2.
  \end{equation*}
  So
  choosing $x=x_-$ is optimal. We obtain
  \begin{equation*}
    \frac{17\kappa}{3\pi^4} x_-^4 -
    \frac{6\kappa}{\pi^2}x_-^2
    +
    \kappa \log x_-+c_0
    \ge -0.753\kappa+c_0.
  \end{equation*}
\end{proof}

\section{A special function}

Let us define
\begin{equation}
  \label{defM}
    M(t)
    =
  t\pi^2\int_{t}^\infty
  \biggl(\frac{\sin \pi  \beta}{\pi \beta}\biggr)^2
  d\beta.
\end{equation}
By integration by parts, we readily discover that
\begin{align*}
  M(t)
  &=
    t\int_t^\infty\frac{1-\cos2\pi \beta}{2\beta^2}d\beta
  =
  \frac{1}{2}+\frac{\sin 2\pi t}{4\pi t}
  -\frac{t}{2\pi}\int_t^{\infty}\frac{\sin 2\pi \beta}{\beta^3}d\beta
  \\&=
  \frac{1}{2}+\frac{\sin 2\pi t}{4\pi t}
  -\frac{\cos 2\pi t}{4\pi^2 t^2}
  +\frac{3t}{4\pi^2}\int_t^{\infty}\frac{\cos 2\pi \beta}{\beta^4}d\beta
  \\&=
  \frac{1}{2}+\frac{\sin 2\pi t}{4\pi t}
  -\frac{\cos 2\pi t}{4\pi^2 t^2}
  -\frac{3\sin 2\pi t}{8\pi^3t^3}
  +\frac{3t}{2\pi^3}\int_t^{\infty}\frac{\sin 2\pi \beta}{\beta^5}d\beta.
\end{align*}
We ran the next GP-Pari script:
\begin{verbatim}
{Moft0(t) =
   return(1/2 + sinc(2*Pi*t)/2 - cos(2*Pi*t)/4/Pi^2/t^2
          - 3*sin(2*Pi*t)/8/Pi^3/t^3);}

{Moft(t) =
   if(t == 0,
      return(0),
      return(Moft0(t)
             + intnum(y = t, [+oo, -2*I*Pi], sin(2*Pi*y)/y^5)*t*3/2/Pi^3));}
	   
s = plothexport("svg", t = 0.1, 4, Moft(t));
write("Moft.svg", s)
\\ and: inkscape --export-type=pdf Moft.svg
\end{verbatim}
to discover that the function $M(t)$ is non-decreasing for
$t\in[0,r_0]$ where $r_0=0.306\,603\,7\cdots$ giving
\begin{equation}
  \label{eq:16}
  M(t)\le 0.6741006.
\end{equation}
The plots are presented in
Figures~\ref{fig20} and~\ref{fig21}.
\begin{figure}
  \centering
  \includegraphics[scale=1]{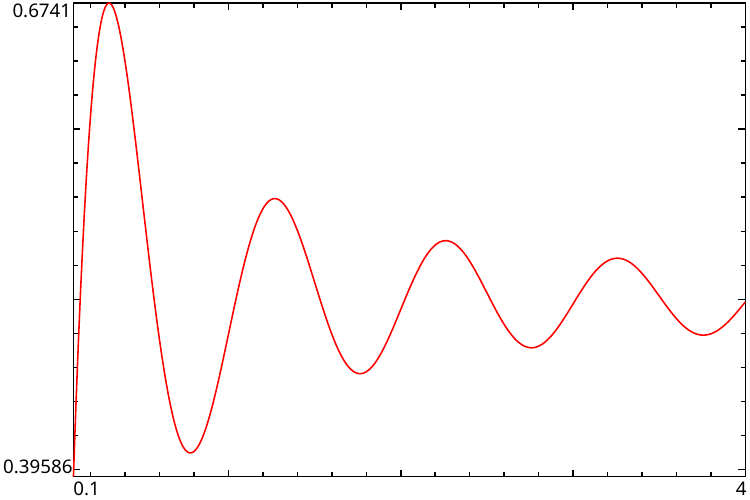}
  \caption{$M(t)$ for $t\in[0.1,4]$}
  \label{fig20}
\end{figure}

\begin{figure}
  \centering
  \includegraphics[scale=1]{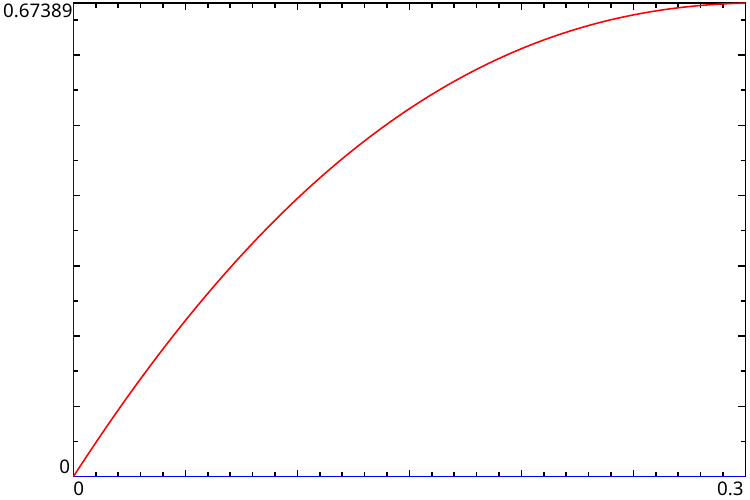}
  \caption{$M(t)$ for $t\in[0,0.3]$}
  \label{fig21}
\end{figure}

We further find that
\begin{equation}
  \label{eq:17}
  M'(t)=\pi^2\int_t^\infty\biggl(\frac{\sin 2\pi
    \beta}{\pi\beta}\biggr)^2d\gamma
  -t\pi^2\biggl(\frac{\sin 2\pi
    t}{\pi t}\biggr)^2
\end{equation}
which shows in passing that $M'(t)$ is bounded over $[0,\infty)$.
To prepare the ground for next section, let us momentarily define
\begin{equation*}
  F_1(t)=(1+t)M\biggl(\frac{1/x^2}{t(1+t)}\biggr).
\end{equation*}
We readily compute that
\begin{equation*}
  F'_1(t)=-(1/x^2)
  \frac{2t+1}{t^2(1+t)^2}
  M'\biggl(\frac{1/x^2}{t(1+t)}\biggr)
  +
  M\biggl(\frac{1/x^2}{t(1+t)}\biggr)
\end{equation*}
and therefore, when $\eta>0$, we obtain:
\begin{equation}
  \label{boundFprime1}
  \int_\eta^1 t(1+|\log t|)|F'_1(t)|dt
  \ll_x
  \int_\eta^1 t(1+|\log t|)\frac{dt}{t^2}\ll (\log \eta)^2.
\end{equation}

\section{Proof  of Theorem~\ref{bettere}}
We provided in Section~\ref{sketch} a proof of a simplified version of
Theorem~\ref{bettere}, proof that countains the most important
ingredients. We aim here at being more precise.
In Lemma~\ref{scalSJQb}, we could be slightly more precise. This
corresponds to replacing the last summand in~\eqref{innerstep} by
\begin{equation}
  \label{eq:5}
  -
    \frac{2Z}{N}\sum_{Q_0\le q\le Q}
    \theta_qg(q)
    R(q)
\end{equation}
where
\begin{equation}
  \label{defRq}
  R(q)=\int_{\delta_q}^{\infty}\int_0^Ne(\beta t)dt
  \overline{\hat{C}_{[0,N],\delta_{\infty}}(\beta)}d\beta,
\end{equation}
with the definition
\begin{equation}
  \label{eq:4}
  \delta_{\infty}=(\log N)^A/N
\end{equation}
and when $\delta_q\delta_\infty^{-1}\le 1$.
The bound for $R(q)$ we have used up to now is
\begin{equation*}
  R(q)\le \frac{1+o(1)}{\pi^2\delta_q}.
\end{equation*}
We find that
\begin{align*}
  R(q)
  &=\int_{\delta_q}^{\infty}\int_0^Ne(\beta t)dt
  \overline{\hat{C}_{[-N/2,N/2],\delta_{\infty}}(\beta)}e(-\frac N2\beta)
    d\beta
  \\&=
  \int_{\delta_q}^{\infty}\int_{-N/2}^{N/2}e(\beta t)dt
  \overline{\hat{C}_{[-N/2,N/2],\delta_{\infty}}(\beta)}
    d\beta
  \\&=
  \int_{\delta_q}^{\delta_\infty}\frac{\sin \pi N\beta}{\pi \beta}
  \delta_\infty^{-1}\biggl(
  (1-|\delta_\infty^{-1}\beta|)\cos \pi N \beta
    +\frac{\hat{J}(\delta_\infty^{-1}\beta)}{\pi \delta_\infty^{-1}\beta}\sin \pi N \beta
  \biggr)
    d\beta
  \\&=
  \delta_\infty^{-1}
  \int_{\delta_\infty^{-1}\delta_q}^{1}\frac{\sin \pi U\beta}{\pi \beta}
  \biggl(
  (1-\beta)\cos \pi U \beta
    +\frac{\hat{J}(\beta)}{\pi \beta}\sin \pi U \beta
  \biggr)
    d\beta
\end{align*}
where $U=N\delta_\infty=(\log N)^A$.
This gives us
\begin{equation*}
  R(q)
  =
  \delta_\infty^{-1}
  \int_{\delta_\infty^{-1}\delta_q}^{1}
  \biggl(
  (1-\beta)\frac{\sin 2\pi U\beta}{2\pi \beta}
    +\hat{J}(\beta)\biggl(\frac{\sin \pi U \beta}{\pi \beta}\biggr)^2
  \biggr)
    d\beta.
  \end{equation*}
  We readily find that
  \begin{equation*}
    \int_{x}^\infty \frac{\sin 2\pi U\beta}{2\pi \beta}d\beta
    =\int_{2\pi Ux}^\infty \frac{\sin \beta}{2\pi \beta}d\beta
    =\frac{\cos 2\pi Ux}{(2\pi)^2Ux}
    -\int_{2\pi Ux}^\infty \frac{\cos \beta}{2\pi \beta^2}d\beta
    \ll 1/(Ux)
  \end{equation*}
  so that
  \begin{align*}
    R(q)
    &=\delta_\infty^{-1}
  \int_{\delta_\infty^{-1}\delta_q}^{1}
  \hat{J}(\beta)\biggl(\frac{\sin \pi U \beta}{\pi \beta}\biggr)^2
  d\beta
  +
  \Ocal\biggl(\frac{1}{U\delta_\infty}+\frac{1}{U\delta_q}\biggr)
    \\&=U\delta_\infty^{-1}
  \int_{U\delta_\infty^{-1}\delta_q}^{U}
  \hat{J}(\beta/U)\biggl(\frac{\sin \pi  \beta}{\pi \beta}\biggr)^2
  d\beta
  +
  \Ocal\biggl(\frac{1}{U\delta_q}\biggr).
  \end{align*}
  At the place where it matters most, we have $\delta_q^{-1}\simeq
  2Q^2\order N$ so that $U\delta^{-1}_\infty \delta_q$ is about
  constant. Having this in mind and on recalling~\eqref{defM}, we
  reduce the above expression of  $R(q)$ to 
  \begin{equation}
    \label{eq:6}
    R(q)\le \pi^2\delta_q^{-1} M(U\delta_\infty^{-1}\delta_q)
    +\Ocal\biggl(\frac{1}{U\delta_q}\biggr),
\end{equation}

We have
\begin{equation*}
  t=t(q)=U\delta_\infty^{-1}\delta_q=\frac{N}{q(q+Q)}=
  \frac{1/x^2}{(q/Q)(1+(q/Q))},
  \quad
  x=Q/\sqrt{N}.
\end{equation*}
We find
that $t(q)\ge 1/(2x^2)$.
On simply taking the maximum, we may use
\begin{equation*}
  R(q)\le \frac{0.6742}{\pi^2\delta_q}.
\end{equation*}
This would already yield and improvement, but we may go one step
further in precision.
We have to bound below
\begin{equation*}
  H=\sum_{ q\le Q}g(q)
  \biggl(1-
  \1_{q\ge Q_0}\frac{2x^2(q/Q)(1+q/Q)}{\pi^2}
  M\biggl(\frac{1/x^2}{(q/Q)(1+q/Q)}\biggr)\biggr)^2
\end{equation*}
on choosing
\begin{equation*}
  \theta_q=1-\1_{q\ge Q_0}\frac{2q(q+Q)}{\pi^2N}
  M\biggl(\frac{1/x^2}{(q/Q)(1+q/Q)}\biggr).
\end{equation*}
We readily find that
\begin{equation*}
  H
  =
  \sum_{ q\le Q}g(q)
  -\frac{4x^2}{\pi^2Q}
  H_1
  +\frac{4x^4}{\pi^4Q}
  H_2
\end{equation*}
where
\begin{align}
  \label{eq:10}
  H_1&=\sum_{ Q_0\le q\le Q}qg(q)
  (1+q/Q)
       M\biggl(\frac{1/x^2}{(q/Q)(1+q/Q)}\biggr),\\
  H_2&=
       \sum_{Q_0\le q\le Q}qg(q)(q/Q)(1+q/Q)^2
  M\biggl(\frac{1/x^2}{(q/Q)(1+q/Q)}\biggr)^2.
\end{align}
We use Lemma~\ref{density+} with $\eta=1/(\log N)^A$.
So we have to compute
\begin{align}
  \label{eq:11}
  K_1(x)&=
       \int_\eta^1(1+t)M\biggl(\frac{1/x^2}{t(1+t)}\biggr)dt,\\
  K_2(x)&=
       \int_\eta^1t(1+t)^2M\biggl(\frac{1/x^2}{t(1+t)}\biggr)^2dt.
\end{align}
The end of Section~\ref{sectionM} (and for
instance~\eqref{boundFprime1}) helps in showing that the relevant
error terms in evaluating $H_1$ and $H_2$ are $\Ocal((\log\log N)^2(\log N)^{\kappa-2})$, and we therefore find that
\begin{equation*}
  \frac{H}{C}
  =
  \bigl(\tfrac12\log N+\log x\bigr)^{\kappa-1}
  \biggl(\tfrac12\log N+\log x
  +c_0
  -\kappa\frac{4x^2}{\pi^2}K_1(x)
  +\kappa\frac{4x^4}{\pi^4}K_2(x)+o(1)\biggr)
\end{equation*}
so
\begin{equation*}
  \frac{2^\kappa H}{C(\log N)^{\kappa-1}}
  =\log N+2\biggl(c_0
  +\kappa\biggl(-\frac{4x^2}{\pi^2}K_1(x)
  +\frac{4x^4}{\pi^4}K_2(x)+\log x\biggr)\biggr)+o(1)
\end{equation*}
provided that
\begin{equation*}
  \forall  t\in(0,1],\quad
  1-\frac{2t(1+t)}{\pi^2}M\biggl(\frac{1/x^2}{t(1+t)}\biggr)\ge0.
\end{equation*}
Since $t(1+t)$ takes any and every values in $(0,2]$, this last
condition is
equivalent to
\begin{equation*}
  \forall  u\in(0,2],\quad
  \frac{\pi^2}{2u}\ge M\biggl(\frac{1/x^2}{u}\biggr)
\end{equation*}
i.e.
\begin{equation*}
  \forall  v\ge 1/(2x^2),\quad
  \frac{\pi^2x^2v}{2}\ge M(v).
\end{equation*}
We choose $x=1.25$ and find that
\begin{equation}
  \label{eq:15}
  x=1.25:\quad -\frac{4x^2}{\pi^2}K_1(x)
  +\frac{4x^4}{\pi^4}K_2(x)+\log x=-0.237\,30\cdots
\end{equation}

The constant for Brun-Titchmarsh becomes
$2(1.332582-0.238)=2.1891\cdots$ while
the constant for the prime twins becomes
$2(A_3-2\times0.238)=11.570\cdots$

\bibliographystyle{plain}

\end{document}